\documentclass[12pt,reqno]{amsart}
\usepackage{amssymb}
\usepackage{amscd}
\usepackage{hyperref}
\usepackage{amssymb}
\usepackage{amscd}
\usepackage{hyperref}

\usepackage{amsxtra}

\usepackage{hyperref}

\usepackage{esint}

\usepackage[shortlabels]{enumitem}
\usepackage{cleveref}

\crefrangeformat{enumi}{items #3#1#4--#5#2#6}

\renewcommand{\frak}{\mathfrak}

\theoremstyle{plain}

\newtheorem{thm}{Theorem}[section]
\newtheorem{lem}[thm]{Lemma}
\newtheorem{cor}[thm]{Corollary}
\newtheorem{prop}[thm]{Proposition}

\theoremstyle{definition}

\newtheorem{defn}[thm]{Definition}

\theoremstyle{remark}

\newtheorem{rem}[thm]{Remark}

\crefname{thm}{theorem}{theorems}
\crefname{lem}{lemma}{lemmas}
\crefname{cor}{corollary}{corollaries}
\crefname{prop}{proposition}{propositions}
\crefname{mainthm}{theorem}{theorems}
\crefname{maincor}{corollary}{corollaries}

\crefname{defn}{definition}{definitions}
\crefname{conj}{conjecture}{conjectures}
\crefname{example}{example}{examples}
\crefname{exercise}{exercise}{exercises}
\crefname{prob}{problem}{problems}
\crefname{quest}{question}{questions}

\crefname{rem}{remark}{remarks}
\crefname{claim}{claim}{claims}
\crefname{axiom}{axiom}{axioms}
\crefname{hyp}{hypothesis}{hypotheses}
\crefname{notation}{notation}{notations}
\crefname{case}{case}{cases}

\numberwithin{equation}{section}

\usepackage{color}
\definecolor{darkgreen}{cmyk}{1,0,1,.2}
\definecolor{m}{rgb}{1,0.1,1}

\newdimen\theight
\def\TeXref#1{%
             \leavevmode\vadjust{\setbox0=\hbox{{\tt
                     \quad\quad  {\small \textrm #1}}}%
             \theight=\ht0
             \advance\theight by \lineskip
             \kern -\theight \vbox to
             \theight{\rightline{\rlap{\box0}}%
             \vss}%
             }}%

\begin{document}
 \title{Transversely symplectic Dirac operators on transversely symplectic foliations} 

 \author[S.~D.~Jung]{Seoung Dal Jung}
 \address{Department of Mathematics\\
 Department of Mathematics\\
 Jeju National University\\
 Jeju 690-756\\
 Republic of Korea}
 \email{sdjung@jejunu.ac.kr}

\subjclass[2010]{53C12; 53C27;57R30}
\keywords{Transversely symplectic foliation, Transversely metaplectic structure, Transversely symplectic Dirac operator}

\begin{abstract}  
We study  transversely  metaplectic structures and transversely symplectic Dirac operators on  transversely symplectic foliations.  And we give the Weitzenb\"ock type formula for transversely symplectic Dirac operators.  Moreover, we estimate the lower bound of the eigenvalues of the transversely symplectic Dirac operator defined by the transverse Levi-Civita connection on transverse  K\"ahler  foliations.
 
\end{abstract}
\maketitle

\renewcommand{\thefootnote}{} \footnote{%
This work was also supported by the National Research Foundation of Korea (NRF) grant funded by the Korea government (MSIP) (NRF-2018R1A2B2002046).
.} \renewcommand{\thefootnote}{\arabic{footnote}} %
\setcounter{footnote}{0}

\section{Introduction}
Symplectic spinor fields were introduced by B. Kostant in \cite{Ko} in the context of geometric quantization. 
They are sections in an $L^2(\mathbb R^n)$-Hilbert space bundle over a symplectic manifold.
In 1995,  K. Habermann \cite{Ha1} defined the symplectic Dirac operator acting on symplectic spinor fields, which is defined in a similar way as the Dirac operator on Riemannian manifolds.  Although the whole construction follows the same procedure as one introduces the Riemannian Dirac operator, using the symplectic structure $\omega$  instead of the Riemannian metric $g$ on $M$,  the underlying algebraical structure of the symplectic Clifford algebra is completely different. For the classical Clifford algebra we have the Clifford multiplication $v\cdot v = -\Vert v \Vert^2$, whereas the symplectic Clifford algebra known as Weyl algebra is given by  the multiplication $u \cdot v - v \cdot u = -\omega_0(u,v)$, where $\omega_0$ is the standard symplectic form on $\mathbb R^{2n}$. From the properties of the Clifford multiplications, the Dirac operators have different properties.  

In this paper,  we study  symplectic spinor fields and symplectic Dirac operators on  transversely symplectic foliations.  Precisely,  we define transversely metaplectic structures  (Section 3) and  give  transversely symplectic Dirac operators $D_{\rm tr} $ and $\tilde D_{\rm tr}$ acting on  symplectic spinor fields (Section 4). The operators $D_{\rm tr} $ and $\tilde D_{\rm tr}$ are not transversely elliptic, and so we define the operator  $\mathcal P_{\rm tr} =\sqrt{-1}[\tilde D_{\rm tr},D_{\rm tr}]$, which is transversely elliptic and formally self-adjoint.   The operator $\mathcal P_{\rm tr}$ is a kind of Laplacian and so it seems to be quite natural  to study the differential operator $\mathcal P_{\rm tr}$ in the symplectic context instead of $D_{\rm tr}^2$.   In section 5, we give the Weitzenb\"ock type formula for the operator $\mathcal P_{\rm tr}$.  The properties of the foliated symplectic spinors  and the special symplectic spinors are given  in section 6 and section 7, respectively. In last section, we study the transversely symplectic Dirac operator defined by the transverse Levi-Civita connection on  transverse K\"ahler foliations. In particular,  we give the lower bound of the eigenvalues of $\mathcal P_{\rm tr}$ on  a transverse K\"ahler foliation of constant holomorphic sectional curvature.

\section{Transversely symplectic foliation}

Let $(M,\mathcal F,\omega)$ be  a transversely symplectic foliation  of codimension $2n$ on a smooth manifold $M$ of dimension $m=p+2n$ with a transversely symplectic form $\omega$.  That is,  $\omega$ is a  closed $2$-form of constant rank $2n$  on $M$ such that   $\ker\omega_x = T\mathcal F_x$ at any point $x\in M$ \cite{AJ,BG,LI}, where $T\mathcal F_x$ is the tangent space of the leave passing through $x$.  Trivially, $\omega$ is a basic form, that is, $i(X)\omega=0$ and $i(X)d\omega=0$ for any vector field $X\in T\mathcal F$.
  
 For examples, contact manifolds and cosymplectic manifolds have  transversely symplectic foliations, which are called as contact flows and cosymplectic flows, respectively \cite{AJ,Pa}.  Also,  a transverse  K\"{a}hler foliation is a transversely symplectic foliation with  a basic  K\"ahler form as a transversely symplectic form.   For more
examples, see  \cite{Cw,LI}.  

 
Let $Q=TM/T\mathcal F$ be the normal bundle of $\mathcal F$. Then the projection $\pi:TM\to Q$ induces a pullback map $\pi^*:\wedge^r Q^*\to \wedge^r T^*M$.  
Let $\Omega_{h}^r(\mathcal F) =\{\phi\in\Omega^r(M)\ |\ i(X)\phi=0 \ \textrm{for any}\ X\in T\mathcal F\}$ and   the linear map $\flat:\Gamma TM\to \Omega_h^1(\mathcal F)$ be defined by $\flat(X)=i(X)\omega$. Trivially, $\ker\flat =T\mathcal F$, and so $\Gamma Q \cong \Omega_h^1(\mathcal F)$. Hence  $\Omega_h^r(\mathcal F)\cong \wedge^r Q^*$ \cite{LI}. 
Clearly,  for any section $\varphi\in \wedge^r Q^*$, $\pi^*(\varphi)\in\Omega_h^r(\mathcal F)$, that is, $i(X)\pi^*(\varphi)=0$ for any $X\in T\mathcal F$. Since $\omega$ is basic, there is a section $\omega_Q\in\wedge^2 Q^*$ such that $\pi^*\omega_Q =\omega$. Thus,  at any point $x\in M$, $(Q_x,\omega_Q)$ is a symplectic vector space \cite{LI}. 

 Let $N\mathcal F$ be a subbundle of $TM$ orthogonal to  $T\mathcal F$ for some Riemannian metric on $M$.  Then $\flat :N\mathcal F \to \Omega_h^1(\mathcal F)$ is an isomorphism and $N\mathcal F\cong Q$. Now, let
 \begin{equation*}
 \mathfrak X_B(\mathcal F) =\flat^{-1}\Omega_B^1(\mathcal F),
 \end{equation*}
 where $\Omega_B^r(\mathcal F)$ is the space of basic $r$-forms.
 Then $X\in \mathfrak X_B(\mathcal F)$ satisfies $[X,Y]\in T\mathcal F$ for any $Y\in T\mathcal F$ \cite{AJ, BC}.
The elements of $\mathfrak X_B(\mathcal F)$ are said to be  {\it basic vector fields} on $M$.  

 Let $\{v_{1},\cdots ,v_{n},w_{1},\cdots
,w_{n}\}$ be a transversely symplectic frame of $\mathcal F$, that is, $v_i,w_i\in \mathfrak X_B(\mathcal F)$ satisfies
\begin{equation*}
\omega(v_i,w_j)=\delta_{ij},\quad \omega(v_i,v_j)=\omega(w_i,w_j)=0.
\end{equation*}
Trivially, if we put $\bar X =\pi(X)$ for any $X\in TM$, then $\{\bar v_i,\bar w_i\}$ is a symplectic frame on $\Gamma Q$ and  $\omega_Q$ is locally expressed as 
\[
\omega_Q =\sum_{i=1}^n \bar v_i^* \wedge \bar w_i^*\;,
\]
where  $\bar v_i^* = -i(\bar w_i)\omega_Q$ and $\bar w_i^*=i(\bar v_i)\omega_Q$ are dual sections.   Any section $s\in \Gamma Q$ is expressed by  $s=\sum_{i=1}^n\{\omega_Q(s, \bar w_i)\bar v_i -\omega_Q(s,\bar v_i)\bar w_i\}$.

Let $\nabla$ be a connection on $Q$. 
  Then the \emph{torsion vector field} $\tau_\nabla$ of $\nabla$ is given  by
\[
\tau_\nabla =\sum_{i=1}^n T_\nabla(v_i,w_i),
\]
where the torsion tensor $T_\nabla$ of $\nabla$ is defined by
\[
T_\nabla(X,Y)=\nabla_X\pi(Y)-\nabla_Y\pi(X)-\pi[X,Y]
\]
for any vector fields $X,Y\in\Gamma TM$.
It is easy to prove that the vector field $\tau_\nabla$ is well-defined; that is, it is independent to the choice of transversely symplectic frames of $\mathcal F$. 
A  \emph{transversely symplectic connection} $\nabla $  on $Q$ is one which
satisfies $\nabla \omega_Q =0$; that is, for all $X\in \Gamma TM$ and $s,t \in \Gamma Q$,
\[
X \omega_Q(s,t) = \omega_Q (\nabla_X s,t) + \omega_Q(s,\nabla_Xt)\;.
\]
There are infinitely many transversely symplectic connections on a transversely symplectic foliation, even infinitely many transversely symplectic connections without torsion (cf. \cite{Va1}).  A  transvesely symplectic connection without torsion (i.e., $\nabla\omega_Q=0$ and $T_\nabla=0$) is said to be {\it transverse Fedosov connection}.  A  transversely symplectic foliation with a transverse Fedosov connection  is said to be {\it transverse Fedosov foliation}.  The transverse Levi-Civita connection on a transverse K\"{a}hler foliation is an example of a  transverse Fedosov connection.
Note that  in contrary to the Levi-Civita connection in Riemannian geometry,  transverse Fedosov connection is not unique. For the study of an ordinary symplectic manifold,  see  \cite{Bl,Ha1,Ha3,To}  and a Fedosov manifold, see \cite{GR,KR}.

Now we define the transversal divergence $\rm {div}_\nabla (s)$ of $s\in \Gamma Q$ with respect to $\nabla$  by
\begin{align*}
{\rm div}_\nabla (s) = {\rm Tr}_Q (Y\to \nabla_Y s),
\end{align*}
that is,   ${\rm div}_\nabla (s)=\sum_{i=1}^n\{\bar v_i^*( \nabla _{v_{i}}s)+\bar w_i^*(\nabla _{w_{i}}s) \}$, equivalently,
\begin{equation}\label{2-2}
{\rm div}_\nabla (s)=\sum_{i=1}^n\{\omega_Q \left( \nabla _{v_{i}}s,\bar w_{i}\right) -\omega_Q \left(
\nabla _{w_{i}}s,\bar v_{i}\right) \}.
\end{equation}
Let $\nu={\frac{1}{n!}}\omega ^{n}$ be the transversal volume form of $
\mathcal{F}$. 
\begin{prop} Let  $\nabla$ be a transversely symplectic connection on $(M,\mathcal{F},\omega )$.  Then, for any   $s\in\Gamma Q$
\begin{equation*}
d(\pi^*s^\flat\wedge \omega ^{n-1})=(n-1)!\{\mathrm{div}_{\nabla
}(s)+\omega_Q (s,\tau_\nabla )\}\nu,
\end{equation*}
where $s^{\flat} =i(s)\omega_Q\in Q^*$.
\end{prop}
\begin{proof} Let $\{v_i,w_i\}$ be a transversely symplectic frame of $\mathcal F$ such that  $\nu (v_{1},w_{1},\cdots ,v_{n},w_{n})=1$. Then it suffices to prove that
\begin{equation}\label{2-3}
d(\pi^*s^\flat\wedge \omega ^{n-1})(v_{1},w_{1},\cdots
,v_{n},w_{n})=(n-1)!\{\mathrm{div}_{\nabla }(s)+\omega_Q (s,\tau_\nabla )\}.
\end{equation}%
Since $\omega $ is closed, we have
\begin{equation*}
d(\pi^*s^\flat\wedge \omega ^{n-1})=d (\pi^*s^\flat)\wedge \omega ^{n-1}.
\end{equation*}%
By a direct calculation, we get
\begin{equation}\label{2-4}
d(\pi^*s^\flat\wedge \omega ^{n-1})(v_{1},w_{1},\cdots
,v_{n},w_{n})=(n-1)!\sum_{i=1}^{n}d(\pi^*s^\flat) (v_{i},w_{i}).
\end{equation}%
From  the symplecity of $\nabla $ and $\pi^*s^\flat(Y)=s^\flat(\bar Y)=\omega_Q(s,\bar Y)$ for any $Y\in\Gamma TM$, we have that, for any $Y,Z \in \Gamma TM$
\begin{equation}\label{2-5}
d(\pi^*s^{\flat})(Y,Z)=\omega_Q (\nabla _{Y}s,\bar Z)-\omega_Q (\nabla _{Z}s, \bar Y)+\omega_Q
(s,T_\nabla(Y,Z)).
\end{equation}%
From (\ref{2-2}) and (\ref{2-5}), we get
\begin{align}\label{2-6}
\sum_{i=1}^{n}d(\pi^*s^{\flat })(v_{i},w_{i})& =\sum_{i=1}^{n}\{\omega_Q (\nabla _{v_{i}}s,\bar w_{i})-\omega_Q (\nabla_{w_{i}}s,\bar v_{i})  +\omega_Q (s,T_\nabla(v_{i},w_{i}))\} \notag\\
& =\mathrm{div}_{\nabla }(s)+\omega_Q (s,\tau_\nabla ).
\end{align}%
From (\ref{2-4}) and (\ref{2-6}),  the proof of (\ref{2-3}) follows.  
\end{proof}
Without loss of generality, we assume that $\mathcal F$ is oriented.  So, given an auxiliary Riemannian metric on $M$ with $N\mathcal F=T\mathcal F^\perp$, there is a unique  $p$-form $\chi_{\mathcal F}$ whose restriction to the leaves is the volume form of the leaves, called the {\it characteristic form} of $\mathcal F$. Now, let $\kappa$ be the corresponding mean curvature form of $\mathcal F$, which are precisely defined in \cite{AJ}.  If $\mathcal F$ is isoparametric (that is, $\kappa$ is basic), then
 $d\kappa=0$ \cite{AJ}. Also,  the Rummler's formula   \cite{AJ} is given by
\begin{equation}  
d\chi_{\mathcal{F}}=-\kappa \wedge \chi_{\mathcal{F}}+\varphi _{0}\;, \label{Rummler formula}
\end{equation}%
where $i(X_1)\cdots i(X_p)\varphi_0=0$ for any vector fields $X_j\in  T\mathcal F$  $(j=1,\cdots,p=\dim T\mathcal F)$.
 Then we have the following theorem. 
\begin{thm} (Transversal divergence theorem) Let $\nabla$ be a transversely symplectic connection on a closed $(M,\mathcal{F},\omega )$. 
Then, for any  $s\in \Gamma_0 Q$,
\begin{equation*}
\int_{M}\mathrm{div}_{\nabla }(s)\mu_M=\int_{M}\omega_Q(\bar\kappa^{\sharp} +\tau_\nabla, s) \mu_M,
\end{equation*}
where $\kappa^{\sharp}=\flat^{-1}(\kappa)$ and $\mu_M=\nu\wedge \chi_\mathcal F$ is the volume form of $M$. 
\end{thm}
\begin{proof}
 Since  the normal degree of $\varphi_0$ is 2,  the normal degree of $\pi^* s^{\flat }\wedge \omega ^{n-1}\wedge\varphi_0$ is  $2n+1$, which is zero.  Hence by the Rummler's formula (\ref{Rummler formula}), 
\begin{align}\label{2-7}
d(\pi^*s^{\flat }\wedge \omega ^{n-1}\wedge \chi_\mathcal F)
=d(\pi^*s^{\flat}\wedge \omega ^{n-1})\wedge \chi_\mathcal F +\pi^*s^{\flat}\wedge\kappa\wedge
\omega ^{n-1} \wedge \chi_\mathcal F.
\end{align}
Now we prove that 
\begin{equation}\label{2-8}
\pi^*s^{\flat}\wedge \kappa \wedge \omega ^{n-1}=(n-1)!\ \omega_Q(s,\bar\kappa^{\sharp}) \nu.
\end{equation}%
In fact, let $\pi^*s^{\flat}\wedge \kappa \wedge \omega ^{n-1}=f\nu $ for any function $f$. Then
\begin{align*}
f& =(\pi^*s^{\flat}\wedge \kappa \wedge \omega ^{n-1})(v_{1},w_1\cdots ,v_n,w_{n})
\\
& =(n-1)!\sum_{i=1}^{n}(\pi^*s^{\flat }\wedge \kappa )(v_{i},w_{i}) \\
& =(n-1)!\sum_{i=1}^n\{\pi^*s^{\flat}(v_{i})\kappa (w_{i})-\pi^*s^{\flat
}(w_{i})\kappa (v_{i})\} \\
& =(n-1)!\sum_{i=1}^n\{\omega_Q (s,\bar v_{i})\kappa (w_{i})-\omega_Q (s,\bar w_{i})\kappa
(v_{i})\} \\
& =(n-1)!\sum_{i=1}^n\{\omega_Q (s,\kappa (w_{i})\bar v_{i}-\kappa (v_{i})\bar w_{i})\} \\
& =(n-1)!\ \omega_Q(s,\bar\kappa^{\sharp})
\end{align*}%
because of  $\kappa =\flat(\kappa^{\sharp})= i(\kappa^{\sharp})\omega$.
From (\ref{2-7}), (\ref{2-8}) and Proposition 2.2, we have
\begin{equation*}
d(\pi^*s^{\flat }\wedge \omega ^{n-1}\wedge \chi_\mathcal F)=(n-1)!\{%
\mathrm{div}_{\nabla }(s)+\omega_Q (s,\tau_\nabla +\bar\kappa^{\sharp})\}\nu
\wedge \chi_\mathcal F.
\end{equation*}%
So  the proof follows from the Stokes' theorem. 
\end{proof}
\begin{cor} Let $\nabla$ be a  transversely symplectic connection  on a closed $(M,\mathcal F,\omega)$.  If $\mathcal F$ is minimal, then for any $s\in \Gamma Q$,
  \begin{equation*}
\int_{M}\mathrm{div}_{\nabla }(s)\mu_M=\int_{M}\omega_Q(\tau_\nabla,s) \mu_M.
\end{equation*}
\end{cor}
\begin{cor} Let  $\nabla$ be a transverse Fedosov connection on a closed $(M,\mathcal F,\omega)$.  Then for any $s\in \Gamma Q$, 
\begin{equation*}
\int_{M}\mathrm{div}_{\nabla }(s)\mu_M=\int_{M}\omega_Q(\bar\kappa^{\sharp},s) \mu_M.
\end{equation*}
In particular,  if $\mathcal F$ is  minimal, then
\begin{equation*}
\int_{M}\mathrm{div}_{\nabla }(s)\mu_M=0.
\end{equation*}

\end{cor}
\begin{rem}Let $(M,\alpha)$ be a contact manifold with a contact form $\alpha$ and let $(M,\eta,\Phi)$ be an almost cosymplectic manifold with a closed 1-form $\eta$ and a closed 2-form $\Phi$, respectively.   Then  the contact  (resp.  cosymplectic) flow $\mathcal F_\xi$, generated by the Reeb vector field $\xi$, is minimal and  transversely symplectic with the transversely symplectic form $\omega=d\alpha$ (resp. $\omega=\Phi$)  \cite{AJ,Pa}. In this case, $\ker\alpha$ and $\ker\eta$ are isomorphic to the normal bundle  of $(M,\alpha)$ and $(M,\eta,\Phi)$, respectively.  Denote by $(M,\mathcal F_\xi,\omega)$ a contact flow or cosymplectic flow.
\end{rem}
\begin{cor} Let  $\nabla$ be a transversely symplectic connection on  a closed $(M,\mathcal F_\xi,\omega)$.   For any  vector field $Y\in \ker \alpha$ (or, $\in\ker\eta$),
\begin{equation*}
\int_{M}\mathrm{div}_{\nabla }(Y)\mu_M=\int_{M}\omega(\tau_\nabla, Y) \mu_M.
\end{equation*}
\end{cor}
\begin{proof}  Since $\mathcal F_\xi$ is minimal, it is trivial from Corollary 2.4.
\end{proof}

Now, we prove the existence of the transversely symplectic connection satisfying $\nabla J=0$ on $(M,\mathcal F,\omega)$ with an $\omega_Q$-compatible almost complex structure $J$ on $Q$,  that is, for any  $s,t \in \Gamma Q$,  $g_Q(s,t) = \omega_Q(s,Jt)$ is an Hermitian metric on $Q$. 
\begin{prop} Let $(M,\mathcal F,\omega,J)$ be a transversely symplectic foliation with an $\omega_Q$-compatible almost complex structure $J$. Then there exists a transversely symplectic connection $\nabla$ such that  $\nabla J=0$ and $\nabla g_Q=0$.
\end{prop}
\begin{proof} Let  $\nabla'$ be an arbitrary transversely symplectic connection. 
We define $\nabla$ by
\begin{equation}\label{2-1}
\nabla_Xs = \nabla'_Xs + \frac12 (\nabla'_XJ)Js
\end{equation}
for any  $X\in \Gamma TM$ and $s\in\Gamma Q$. It is easily proved that $\nabla $ is transversely symplectic. From (\ref{2-1}), we get
\begin{align*}
\nabla_X Js=\frac12\nabla'_X Js +\frac12 J\nabla'_Xs
\end{align*}
and
\begin{align*}
J\nabla_Xs = \frac12 J\nabla'_Xs +\frac12 \nabla'_XJs.
\end{align*}
Hence $\nabla_XJs = J\nabla_Xs$, which implies 
$\nabla J=0$.  Next, since $\nabla J$ and $\nabla\omega_Q=0$, 
\begin{equation*}
(\nabla_X g_Q)(s,t) = (\nabla_X\omega_Q)(s,Jt)+\omega_Q(s,(\nabla_XJ)t)=0
\end{equation*}
for any  $X\in \Gamma TM$ and $s,t\in \Gamma Q$, which proves $\nabla g_Q=0$.
\end{proof}
 
\section{Transversely metaplectic structure}
Let $(M,\mathcal F,\omega)$ be a  transversely symplectic foliation  of codimension $2n$.   Let $P_{Sp}(Q)$ be  the principal $Sp(n,\mathbb R)$-bundle over $M$ of all symplectic frames on the normal bundle $Q$, where $Sp(n,\mathbb R)$ is the symplectic group (i.e., the group of all automorphisms of $\mathbb R^{2n}$ which preserve the standard symplectic form $\omega_0$ on $\mathbb R^{2n}$).  Since the first homotopy group of $Sp(n,\mathbb R)$ is isomorphic to $\mathbb Z$, there exists a unique connected double covering of $Sp(n,\mathbb R)$, which is known as {\it metaplectic group} $Mp(n,\mathbb R)$ \cite{Ha3}.  Let $\rho : Mp(n,\mathbb R)\to Sp(n,\mathbb R)$ be the two-fold covering map \cite{Ha1}.   A {\it transversely metaplectic structure} on $M$ is a principal $Mp(n,\mathbb R)$-bundle $\tilde P_{Mp}(Q)$ over $M$ together with a bundle morphism $F:\tilde P_{Mp}(Q)\to P_{Sp}(Q)$ which is equivariant with respect to $\rho$ (precisely, see \cite{Ha1, Ha2}).   A transversely symplectic foliation admits a transversely metaplectic structure if and only if the  second Stiefel-Withney class in  $H^2(Q,\mathbb Z_2)$ (the second ${\check C}$ech cohomology group of the normal bundle $Q$) vanishes (cf. \cite{LM}). 

 From now on, we  consider a transversely symplectic foliation with a fixed transversely metaplectic structure $\tilde P_{Mp}(Q)$.
 Let $\frak m:Mp(n,\mathbb R)\to U(L^2(\mathbb R^n))$ be the  metaplectic representation (Segal-Shale-Weil representation) \cite{Ka} which satisfies
 \begin{equation} \label{3-0}
 \frak m(g)\circ \frak r_S(v,t) = \frak r_S (\rho(g)v,t)\circ\frak m(g) 
 \end{equation}
 for all $g\in Mp(n,\mathbb R)$ and $(v,t)\in H(n)=\mathbb R^{2n}\times \mathbb R$, where $\frak r_S: H(n) \to U(L^2(\mathbb R^n))$ is the Schr\"odinger representation, $H(n)$ is the Heizenberg group  and  $U(L^2(\mathbb R^n))$ is the unitary group on $L^2(\mathbb R^n)$ of square integrable functions on $\mathbb R^n$ \cite{Ha3}.  
  The representation $\frak m$ stabilizes the Schwartz space $\mathcal  S(\mathbb R^n)\subset L^2(\mathbb R^n)$ of rapidly decreasing smooth functions on $\mathbb R^n$, that is, $\mathcal S(\mathbb R^n)$ is  $\frak m$-invariant \cite{Ka}. 
 The symplectic Clifford multiplication $\mu_0:\mathbb R^{2n}\otimes L^2(\mathbb R^n)\to L^2(\mathbb R^n)$ is defined by
 \begin{equation}\label{3-1}
 \mu_0(v\otimes f)=\sigma(v)f,
 \end{equation}
 where $\sigma:\mathbb R^{2n} \to  {\rm End}(L^2(\mathbb R^n))$ is the linear map such that $\sigma(a_j)=\sqrt{-1}x_j$ and $\sigma(b_j)={\partial\over\partial x_j} (j=1,\cdots,n)$ \cite{Ha3}. Here $\{a_i,b_i\}$ is the symplectic frame on $\mathbb R^{2n}$ with respect to the standard symplectic form $\omega_0$. For any $v,w\in \mathbb R^{2n}$, 
 \begin{equation}\label{3-1-1}
 \sigma(v)\sigma(w) -\sigma(w)\sigma(v)= -\sqrt{-1}\omega_0(v,w).
 \end{equation} 
 By using the metaplectic representation $\frak m$, we define the Hilbert bundle $Sp(\mathcal F)$ associated with the transversely metaplectic structure $\tilde P_{Mp}(Q)$ by
\begin{align}
Sp(\mathcal F)=\tilde P_{Mp}(Q) \times_\frak m L^2 (\mathbb R^n),
\end{align}
which is called  a {\it foliated symplectic spinor bundle} over $M$.  A {\it foliated symplectic spinor field} on $(M,\mathcal F,\omega)$ is a section $\varphi=[p,f]\in\Gamma  Sp(\mathcal F)$, the space of all smooth sections of  $Sp(\mathcal F)$ such that $[pg,f]=[p,\frak m(g^{-1})f]$ for any $g\in Mp(n,\mathbb R)$ and $f\in\mathcal S(\mathbb R^n)$.

Now, if we consider the normal bundle $Q$ as $Q=\tilde P_{Mp}(Q) \times_\rho \mathbb R^{2n}$, then a section in $Q$ can be written as equivalence classes $[p,v]$ of pairs $(p,v)\in \tilde P_{Mp}(Q)\times \mathbb R^{2n}$. Hence we can define the {\it symplectic Clifford multiplication} $\mu_Q:Q\otimes \Gamma Sp(\mathcal F)\to \Gamma Sp(\mathcal F)$ on $\Gamma Sp(\mathcal F)$  by
\begin{equation}\label{3-2}
\mu_Q([p,v]\otimes [p,f])=[p,\sigma(v)f]
\end{equation}
for any smooth section $[p,f]\in \Gamma Sp(\mathcal F)$.  Denote by 
\begin{equation}\label{3-3}
\mu_Q(s\otimes\varphi)=s\cdot\varphi
\end{equation}
 for any $s\in \Gamma Q$ and $\varphi\in \Gamma Sp(\mathcal F)$.
 By the property (\ref{3-1-1}) of $\sigma$, we have
\begin{equation}\label{3-4}
(s\cdot t- t\cdot s)\cdot\varphi =-\sqrt{-1}\omega_Q (s,t)\varphi
\end{equation}
for any  $s,t\in \Gamma Q$ and $\varphi\in \Gamma Sp(\mathcal F)$ \cite {Ha2,Ha3}.  Let  $<\cdot , \cdot >$ be a  canonical Hermitian scalar product  on $Sp(\mathcal F)$ given by the $L^2(\mathbb R^n)$-scalar product on the fibers. That is, for any $\varphi_1=[p,f_1],\varphi_2=[p,f_2]\in Sp(\mathcal F)$, we define $<\varphi_1,\varphi_2> = <f_1,f_2>$, where $<f_1,f_2>$ is the $L^2$-product of the functions $f_1,f_2\in L^2(\mathbb R^n)$.  For any $v\in\mathbb R^{2n}$ and $f_1, f_2 \in L^2(\mathbb R^n)$,  we get \cite[Lemma 1.4.1(2)]{Ha3}
\begin{equation*}
<\sigma(v)f_1,f_2> =-<f_1,\sigma(v)f_2>,
\end{equation*} 
which yields
\begin{equation}\label{3-6}
<s\cdot\varphi,\psi> =-<\varphi,s\cdot\psi>
\end{equation} 
for any $s\in \Gamma Q$ and $\varphi\in \Gamma Sp(\mathcal F)$.
 Let $\nabla$ be a spinor derivative on $Sp(\mathcal F)$ which is induced by a  transversely symplectic connection $\nabla$ on $Q$.  Similar to an ordinary manifold (see \cite{Ha1} or \cite[Proposition 3.2.6]{Ha3}),  $\nabla$ is locally given by
\begin{equation}\label{3-5}
\nabla_X\varphi =X(\varphi) +{1\over 2\sqrt{-1}}\sum_{j=1}^n\{\bar w_j\cdot\nabla_X \bar v_j -\bar v_j\cdot\nabla_X\bar w_j\}\cdot\varphi   
\end{equation}
for any vector field $X\in \Gamma TM$, where $X(\varphi) = [p,X(f)]$ for $\varphi=[p,f]\in\Gamma Sp(\mathcal F)$.  
 Then we have the following properties on $\Gamma Sp(\mathcal F)$ :
\begin{align}
 &\nabla_X (s\cdot\varphi)=(\nabla_X s)\cdot\varphi + s\cdot\nabla_X\varphi,\label{3-7}\\
&X<\varphi,\psi>=<\nabla_X\varphi,\psi> +<\varphi,\nabla_X\psi>\label{3-8}
\end{align}
for any $s\in\Gamma Q, \ X\in \Gamma TM$ and $\varphi,\psi\in \Gamma Sp(\mathcal F)$  \cite{Ha1,Ha3}.
 And the curvature tensor $R^S$ of the spinor derivative $\nabla$ on $\Gamma Sp(\mathcal F)$ is given by
\begin{align}\label{3-9}
R^S(X,Y)\varphi&={\sqrt{-1}\over 2}\sum_{i=1}^n \{\bar v_i\cdot R^\nabla(X,Y)\bar w_i-\bar w_i\cdot R^\nabla(X,Y)\bar v_i\}\cdot\varphi\\
&={\sqrt{-1}\over 2}\sum_{i=1}^n \{R^\nabla(X,Y)\bar w_i\cdot \bar v_i-R^\nabla(X,Y)\bar v_i\cdot \bar w_i\}\cdot\varphi
\end{align}
for any vector fields $X,Y\in \Gamma TM$, 
where $R^\nabla(X,Y)=[\nabla_X,\nabla_Y] -\nabla_{[X,Y]}$ is the curvature tensor of $\nabla$  on $Q$  \cite{Ha1}.
Moreover,  the curvature tensor  $R^\nabla$ satisfies the following.
\begin{lem}  Let  $(M,\mathcal F,\omega)$ be a transversely symplectic foliation with a transversely symplectic connection $\nabla$. Then
for any $X,Y\in\Gamma TM$ and $s,t\in\Gamma Q$, 
\begin{equation}\label{3-11}
\omega_Q(R^\nabla(X,Y)s,t) = \omega_Q(R^\nabla(X,Y)t,s).
\end{equation}
Moreover, if  $\nabla$ satisfies  $\nabla J=0$ for an $\omega_Q$-compatible almost complex structure $J$ on $Q$, then
\begin{equation}\label{3-12}
\omega_Q(R^\nabla(X,Y)Js,Jt) =\omega_Q(R^\nabla(X,Y)s,t).
\end{equation}
\end{lem}
\begin{proof}  The proofs are easy.
\end{proof}

\section{Transversely symplectic Dirac operators}
Let $(M,\mathcal F,\omega,\nabla)$ be a transversely symplectic foliation with 
 fixed a transversely metaplectic structure and  a transversely symplectic connection $\nabla$. 
  Let  
\begin{equation*}
\nabla_{\rm tr}=\pi\circ\nabla : \Gamma Sp(\mathcal F)\overset{\nabla}\to \Gamma (TM^*\otimes Sp(\mathcal F))\overset{\pi}\to \Gamma (Q^*\otimes Sp(\mathcal F)),
\end{equation*}
where $\nabla$ is the spinorial derivative induced from the transversely symplectic connection  on $Q$.  
  Then we  define the operator $D_{\rm tr}'$ by 
\begin{equation}\label{4-1}
D_{\rm tr}'=\mu_Q\circ\nabla_{\rm tr}:\Gamma Sp(\mathcal F)\overset{\nabla_{\rm tr}} \to \Gamma (Q^*\otimes Sp(\mathcal F))\overset{\omega_Q}\cong \Gamma(Q\otimes Sp(\mathcal F))\overset{\mu_Q}\to \Gamma Sp(\mathcal F),
\end{equation}
where $Q^*\cong Q$ by the symplectic structure $\omega_Q$ such that $i(s)\omega_Q\cong s$ for any $s\in \Gamma Q$.
If we identify $Q^*$ and $Q$ by the  Riemannian metric $g_Q$ associated to $\omega_Q$, then we obtain a second operator $\tilde D_{\rm tr}'$ by
\begin{equation}\label{4-2}
\tilde D_{\rm tr}'=\mu_Q\circ\nabla_{\rm tr}:\Gamma Sp(\mathcal F)\overset{\nabla_{\rm tr}}  \to \Gamma (Q^*\otimes Sp(\mathcal F))\overset{g_Q}\cong \Gamma(Q\otimes Sp(\mathcal F))\overset{\mu_Q}\to \Gamma Sp(\mathcal F).
\end{equation}
  From (\ref{4-1}),  $D_{\rm tr}'$ is locally given by
    \begin{align*}
    D_{\rm tr}'\varphi &=\mu_Q(\nabla_{\rm tr}\varphi) =\mu_Q(\sum_{i=1}^n \{\bar v_i^*\otimes\nabla_{v_i}\varphi + \bar w_i^*\otimes\nabla_{w_i}\varphi\}).
    \end{align*}
   Since $\bar v_i^* =-i(\bar w_i)\omega_Q\cong -\bar w_i$ and $\bar w_i^*=i(\bar v_i)\omega_Q\cong \bar v_i$,  we have from (\ref{3-2}) and (\ref{3-3}), 
  \begin{equation}\label{4-3}
D_{\rm tr}'\varphi=\sum_{i=1}^n \{\bar v_i\cdot\nabla_{w_i}\varphi-\bar w_i\cdot\nabla_{v_i}\varphi\}.
\end{equation}
Let $J$ be an $\omega_Q$-compatible almost complex structure on $Q$.  Since $\bar v_i^* (s)=-g_Q(J\bar w_i,s)$ and $\bar w_i^*(s)=g_Q(J\bar v_i,s)$ for any $s\in\Gamma Q$,  from (\ref{4-2}), $\tilde D_{\rm tr}'$ is locally  given by
\begin{equation}\label{4-4}
\tilde D_{\rm tr}'\varphi=\sum_{i=1}^n\{J\bar v_i\cdot\nabla_{w_i}\varphi -J\bar w_i\cdot\nabla_{v_i}\varphi\}.
\end{equation}

\begin{rem}
The definitions of $D_{\rm tr}'$ and $ \tilde D_{\rm tr}'$ depend on a choice of a transversely symplectic connection on $Q$ as well as on a choice of a transversely metaplectic structure of $\mathcal F$. Moreover,  $\tilde D_{\rm tr}'$  also depends on an arbitrary almost complex structure $J$ compatible with $\omega_Q$ (cf. \cite{Ha1,Ha3,Ka,Ko}). 
\end{rem}
In what follows,  we fix  a transversely metaplectic structure and an $\omega_Q$-compatible almost complex structure $J$ on $(M,\mathcal F,\omega)$. 
From (\ref{3-6}) $\sim$ (\ref{3-8}),  we get 
\begin{align*}
<D_{\rm tr}'\varphi,\psi>&=<\varphi, D_{\rm tr}'\psi>\\
& +\sum_{i=1}^n  \{v_i<\varphi,\bar w_i\cdot\psi> -w_i<\varphi,\bar v_i\cdot\psi>+<\varphi,(\nabla_{w_i}\bar v_i - \nabla_{v_i}\bar w_i)\cdot\psi>\}
\end{align*}
for any $\varphi, \psi \in \Gamma Sp(\mathcal F)$.
If we choose $s\in \Gamma Q$ such that $ \omega_Q(s,t) =<\varphi,t\cdot\psi>$ for any $t\in \Gamma Q$, then $\nabla\omega_Q=0$ implies 
\begin{align*}
&\sum_{i=1}^n \{v_i <\varphi,\bar w_i\cdot\psi>- w_i<\varphi,\bar v_i\cdot\psi> +<\varphi, (\nabla_{w_i}\bar v_i - \nabla_{v_i}\bar w_i)\cdot\psi>\}\\
&=\sum_{i=1}^n\{v_i\omega_Q(s,\bar w_i) - w_i\omega_Q(s,\bar v_i)+\omega_Q(s,\nabla_{w_i}\bar v_i -\nabla_{v_i}\bar w_i)\}\\
&=\sum_{i=1}^n \{\omega_Q(\nabla_{v_i}s,\bar w_i)-\omega_Q(\nabla_{w_i}s,\bar v_i)\}\\
&={\rm div}_\nabla(s)
\end{align*}
and so
\begin{equation}\label{4-5}
<D_{\rm tr}'\varphi,\psi>=<\varphi,D_{\rm tr}'\psi> + {\rm div}_\nabla(s).
\end{equation}
If we  integrate (\ref{4-5}) with the transversal divergence theorem (Theorem 2.3), then
\begin{align*}
\int_M <D_{\rm tr}'\varphi,\psi> = \int_M <\varphi,D_{\rm tr}'\psi -(\bar\kappa^\sharp+\tau_\nabla)\cdot\psi>.
\end{align*}
Hence  the formal adjoint operator $D_{\rm tr}'^*$ is given by
\begin{align*}
D_{\rm tr}'^*\varphi=D_{\rm tr}'\varphi -(\bar\kappa^{\sharp}+\tau_\nabla)\cdot\varphi,
\end{align*}
which implies that $D_{\rm tr}'$ is not formally self-adjoint.   So if we put $D_{\rm tr}$ by 
\begin{align}\label{4-7}
D_{\rm tr}\varphi=D_{\rm tr}'\varphi -\frac12 (\bar\kappa^{\sharp}+\tau_\nabla)\cdot\varphi,
\end{align}
then $D_{\rm tr}$  is formally self-adjoint.  This operator $D_{\rm tr}$ is said to be  {\it transversely symplectic Dirac operator} of $\mathcal F$. 
   
   In addition,  if $\nabla$ satisfies $\nabla J=0$,  then
\begin{align*}
<\tilde D_{\rm tr}'\varphi,\psi>&=<\varphi,\tilde D_{\rm tr}'\psi>\\
& + \sum_{i=1}^n\{v_i<\varphi,J\bar w_i\cdot\psi>-w_i<\varphi,J\bar v_i\cdot\psi>+<\varphi,J (\nabla_{w_i}\bar v_i -\nabla_{v_i}\bar w_i)\cdot\psi>\}.
\end{align*}
If we choose $s\in\Gamma Q$ such that $\omega_Q(s,t) = <\varphi,Jt\cdot\psi>$ for any $t\in \Gamma Q$,  then
\begin{align}\label{4-8}
<\tilde D_{\rm tr}'\varphi,\psi>=<\varphi,\tilde D_{\rm tr}'\psi> +{\rm div}_\nabla (s).
\end{align}
Hence  by integrating (\ref{4-8}) together with  the transversal divergence theorem (Theorem 2.3),  the formal adjoint operator $\tilde D_{\rm tr}'^*$ is given by
\begin{align*}
\tilde D_{\rm tr}'^*\varphi = \tilde D_{\rm tr}'\varphi -J(\bar\kappa^{\sharp}+\tau_\nabla)\cdot\varphi,
\end{align*}
which implies that $\tilde D_{\rm tr}'$ is also not formally self-adjoint. Therefore, if we put $\tilde D_{\rm tr}$ by
\begin{align}\label{4-10}
\tilde D_{\rm tr} \varphi = \tilde D_{\rm tr}'\varphi -\frac12 J(\bar\kappa^{\sharp}+\tau_\nabla)\cdot\varphi,
\end{align}
then $\tilde D_{\rm tr}$ is formally self-adjoint.     
    The operator $\tilde D_{\rm tr}$ is also said to be   {\it second  transversely symplectic Dirac operator} of $\mathcal F$. 
Hence we have the following theorem.
\begin{thm} Let $(M,\mathcal F,\omega,\nabla)$ be a transversely symplectic foliation with a transversely metaplectic structure and a transversely symplectic connection $\nabla$  on a closed, connected  manifold $M$. Then  $D_{\rm tr}$ is   formally self-adjoint. In particular, if  $\nabla J=0$, then $\tilde D_{\rm tr}$ is also formally self-adjoint.
\end{thm}
Let $\xi\in T_x^*M$ and $f$ be a smooth function on $M$ such that $df_x =\xi$ and $f(x)=0$. 
 And let $\tilde\varphi\in \Gamma Sp(\mathcal F)$ and $\varphi\in Sp_x(\mathcal F)$ such that $\tilde\varphi (x)=\varphi$. Then the principal symbols of  $D_{\rm tr}$ and $\tilde D_{\rm tr}$  are given by
\begin{align*}
\sigma(D_{\rm tr})_{\xi}\varphi &= \sum_{i=1}^n \{df(w_i)\bar v_i - df(v_i)\bar w_i\}\cdot\varphi,\\
\sigma(\tilde D_{\rm tr})_{\xi}\varphi &=\sum_{i=1}^n \{df(w_i)J\bar v_i - df(v_i)J\bar w_i\}\cdot\varphi,
\end{align*}
respectively.  Precisely, if $\xi\in Q^*$, then  
\begin{align}
\sigma(D_{\rm tr})_\xi\varphi = \bar\xi^\sharp\cdot\varphi,\quad \sigma(\tilde D_{\rm tr})_\xi\varphi = J\bar\xi^\sharp\cdot\varphi,
\end{align}
where $\xi = i(\xi^\sharp)\omega = g_Q(J\xi^\sharp,\cdot)$. 
If  $\xi\in (T_x\mathcal F)^*$, then $df(v_i)=df(w_i)=0$. So $\sigma(D_{\rm tr})_\xi =\sigma(\tilde D_{\rm tr})_\xi=0$.  
This implies that the principal symbols are not isomorphisms. And so $D_{\rm tr}$ and $\tilde D_{\rm tr}$ are not  elliptic. Moreover, they are not transversally elliptic because $\bar \xi^\sharp\cdot\varphi =0$ does not implies $\bar \xi^\sharp=0$.

Now, we introduce a new transversally elliptic operator of second order which is of Laplace type.
\begin{defn} Let $\mathcal P_{\rm tr}:\Gamma Sp(\mathcal F) \to \Gamma Sp(\mathcal F)$ be the second order operator defined by 
\begin{equation*}
\mathcal P_{\rm tr}=\sqrt{-1}[\tilde D_{\rm tr},D_{\rm tr}].
\end{equation*}
\end{defn}
Trivially, $\mathcal P_{\rm tr}$ is formally self-adjoint. Moreover,  if  $\xi \in Q^*$ at $x$, then the principal symbol  of $\mathcal P_{\rm tr}$ is $\sigma(\mathcal P_{\rm tr})_\xi=\omega_Q(J\bar \xi^\sharp,\bar \xi^\sharp) = -g_Q(\bar \xi^\sharp,\bar \xi^\sharp)$.  If  $\xi\in (T_x\mathcal F)^*$, then $\sigma(\mathcal P_{\rm tr})_\xi =0$. So, $\mathcal P_{\rm tr}$ is a transversally elliptic operator of Laplace type. 


\section{The Weitzenb\"ock formula}
Let $(M,\mathcal F,\omega,\nabla)$ be a transversely symplectic foliation with 
fixed a transversely metaplectic structure and  a transversely symplectic connection $\nabla$.  In this section,
we study the Weitzenb\"ock formula for the operator $\mathcal P_{\rm tr}$ on $M$.

Let $\{e_1,\cdots,e_{2n}\}$ be a unitary basic  frame  in $Q$,  that is, $e_j\in\mathfrak X_B(\mathcal F) (j=1,\cdots,2n)$ satisfies $g_Q(\bar e_i,\bar e_j)=\delta_{ij}$ and $\bar e_{n+i}=J\bar e_i (i=1,\cdots,n)$. 
Then we have the following.  
 \begin{lem}  The operators $D_{\rm tr}'$ and $\tilde D_{\rm tr}'$ are also given by 
\begin{equation*}
D_{\rm tr}'\varphi=-\sum_{i=1}^{2n} J\bar e_i\cdot\nabla_{e_i}\varphi,\quad \tilde D_{\rm tr}'\varphi=\sum_{i=1}^{2n} \bar e_i\cdot\nabla_{e_i}\varphi.
\end{equation*}
\end{lem}
\begin{proof} From (\ref{4-3}) and (\ref{4-4}), the proof follows.
\end{proof}
Let $\nabla_{\rm tr}^*:\Gamma (Q^*\otimes Sp(\mathcal 
F))\to \Gamma Sp(\mathcal F)$ be the formal adjoint operator of the spinor derivative $\nabla_{\rm tr}$. 
\begin{lem} Let $M$ be  a closed manifold.  Then for any $\Psi\in\Gamma(Q^* \otimes Sp(\mathcal F))$, 
\begin{align*}
\nabla_{\rm tr}^*\Psi=-{\rm tr}_Q(\nabla_{\rm tr}\Psi)+\Psi(J(\bar\kappa^\sharp+\tau_\nabla)).
\end{align*}
\end{lem}
\begin{proof} Note that for any $\Phi,\Psi\in\Gamma(Q^*\otimes Sp(\mathcal F))$,
\begin{align*}
<\Phi,\Psi>=\sum_{i=1}^{2n}<\Phi(\bar e_i),\Psi(\bar e_i)>.
\end{align*}
 Then  for any $\varphi\in \Gamma Sp(\mathcal F)$,
\begin{align*}
<\nabla_{\rm tr}\varphi,\Psi>&=\sum_{i=1}^{2n}<\nabla_{e_i}\varphi,\Psi(\bar e_i)>\\
&=\sum_{i=1}^{2n}\{ e_i<\varphi,\Psi(\bar e_i)> -<\varphi,\nabla_{e_i}\Psi(\bar e_i)>\}.
\end{align*}
If we choose $s\in \Gamma Q$ such that $ g_Q(s,t) =<\varphi, \Psi(t)>$ for any $t\in \Gamma Q$, then 
\begin{equation*}
{\rm div}_\nabla (s)=\sum_{i=1}^{2n} \{e_i< \varphi,\Psi(\bar e_i)> +< \varphi,{\rm div}_\nabla(\bar e_i)\Psi(\bar e_i)>\}.
\end{equation*}
Hence 
\begin{equation}\label{5-1}
<\nabla_{\rm tr}\varphi,\Psi> ={\rm div}_\nabla (s)-\sum_{i=1}^{2n}<\varphi, \nabla_{e_i}\Psi(\bar e_i) +{\rm div}_\nabla (\bar e_i) \Psi(\bar e_i)>.
\end{equation}
On the other hand, by the divergence theorem (Theorem 2.3), we have
\begin{align*}
\int_M {\rm div}_\nabla (s)=\int_M \omega_Q(\bar\kappa^\sharp+\tau_\nabla,s) =\int_M g_Q(J(\bar\kappa^\sharp+\tau_\nabla),s)=\int_M <\varphi,\Psi(J(\bar\kappa^\sharp+\tau_\nabla))>.
\end{align*}
Hence by integrating (\ref{5-1}), 
\begin{align*}
\int_M <\varphi, \nabla_{\rm tr}^*\Psi>&=\int_M<\nabla_{\rm tr}\varphi,\Psi>\\
&=\int_M <\varphi,\Psi(J(\bar\kappa^\sharp+\tau_\nabla))>-\sum_{i=1}^{2n}\int_M<\varphi,\nabla_{e_i}\Psi(\bar e_i) +{\rm div}_\nabla (\bar e_i) \Psi(\bar e_i)>\\
&=\int_M<\varphi,\Psi(J(\bar\kappa^\sharp+\tau_\nabla))> -\int_M <\varphi, {\rm tr}_Q (\nabla_{\rm tr}\Psi)> ,
\end{align*}
which completes the proof.
\end{proof}
From Lemma 5.2, we have the following.
\begin{prop} For any spinor field $\varphi\in\Gamma Sp(\mathcal F)$, we have
\begin{align*}
\nabla_{\rm tr}^*\nabla_{\rm tr}\varphi=-\sum_{i=1}^{2n}\{\nabla_{e_i}\nabla_{e_i}\varphi+{\rm div}_\nabla(\bar e_i)\nabla_{e_i}\varphi \}+\nabla_{J(\kappa^\sharp+\tau_\nabla)}\varphi.
\end{align*}
\end{prop}
Now, we put that for any $s \in \Gamma Q$,
\begin{equation}\label{5-2}
P(s)=\sum_{i=1}^{2n}\bar e_i\cdot\nabla_{Je_i}s,\quad
\tilde P(s)=\sum_{i=1}^{2n}\bar e_i\cdot\nabla_{e_i}s.
\end{equation}
By a direct calculation, we have the following lemmas.
\begin{lem} For any $s \in\Gamma Q$ and $\varphi\in \Gamma Sp(\mathcal F)$, we have
\begin{align}
D_{\rm tr}(s\cdot\varphi)&= s\cdot D_{\rm tr}\varphi +P(s)\cdot \varphi -\sqrt{-1}\nabla_{s}\varphi -\frac{\sqrt{-1}}2\omega_Q(s,\bar\kappa^\sharp+\tau_\nabla)\varphi,\label{5-3}\\
\tilde D_{\rm tr}(s\cdot\varphi)&=s\cdot \tilde D_{\rm tr}\varphi+\tilde P(s)\cdot\varphi+\sqrt{-1}\nabla_{Js}\varphi +\frac{\sqrt{-1}}2\omega_Q(Js,\bar\kappa^\sharp+\tau_\nabla)\varphi.\label{5-4}
\end{align}
\end{lem}
\begin{proof}  
From (\ref{4-7}),  (\ref{4-10}) and Lemma 5.1, we have
\begin{align*}
D_{\rm tr}(s\cdot\varphi)=&-\sum_{i=1}^{2n} J\bar e_i\cdot \nabla_{e_i}(s\cdot\varphi)-\frac12(\bar\kappa^\sharp+\tau_\nabla)\cdot s\cdot\varphi\\
=&-\sum_{i=1}^{2n}J\bar e_i\cdot\nabla_{e_i}s\cdot\varphi -\sum_{i=1}^{2n}s\cdot J\bar e_i\cdot\nabla_{e_i}\varphi +\sqrt{-1}\sum_{i=1}^{2n}\omega_Q(J\bar e_i,s)\nabla_{e_i}\varphi\\
&  -\frac12(\bar\kappa^\sharp+\tau_\nabla)\cdot s \cdot\varphi\\
=& P(s)\cdot\varphi + s\cdot D_{\rm tr}\varphi -\sqrt{-1} \nabla_{s} \varphi +\frac12\{ s\cdot (\bar\kappa^\sharp +\tau_\nabla) - (\bar\kappa^\sharp+\tau_\nabla)\cdot s\}\cdot\varphi,
\end{align*}
which proves (\ref{5-3}).  Similarly, (\ref{5-4}) is proved.  
\end{proof}

\begin{lem}  (cf. \cite[Lemma 5.2.5]{Ha3}) For any $s\in \Gamma Q$, we  have
\begin{align*}
P(s)+\tilde P(Js)=&-P(J)(Js)-\sqrt{-1} {\rm div}_\nabla(s)\\
&+\sum_{i,j=1}^{2n}\{e_i\omega_Q(\bar e_j,Js)-e_j\omega_Q(\bar e_i,Js)\}\bar e_i\cdot J\bar e_j\\
&-\sum_{i,j=1}^{2n}\omega_Q(T_\nabla(e_i,e_j)+\pi[e_i,e_j],Js)\bar e_i\cdot J\bar e_j,
\end{align*}
where $P(J)(s)=\sum_{i=1}^{2n}(\nabla_{Je_i}J)(s)\cdot \bar e_i$.
\end{lem}

\begin{thm} $($Weitzenb\"ock formula$)$  On a transversely symplectic foliation  $(M,\mathcal F,\omega,\nabla)$,  we have the following Weitzenb\"ock formula; for any $\varphi \in \Gamma Sp(\mathcal F)$
\begin{align*}
\mathcal P_{\rm tr}\varphi=& \nabla_{\rm tr}^*\nabla_{\rm tr}\varphi+\sqrt{-1}F(\varphi)-\frac14|\bar\kappa^\sharp+\tau_\nabla|^2\varphi+{\sqrt{-1}\over 2}\{P(J(\bar\kappa^\sharp+\tau_\nabla))-\tilde P(\bar\kappa^\sharp+\tau_\nabla)\}\cdot\varphi\\
&+\sqrt{-1}\sum_{i=1}^{2n}P(J)(J\bar e_i)\cdot\nabla_{e_i}\varphi+\sqrt{-1}\sum_{i,j=1}^{2n}\bar e_i\cdot J\bar e_j\cdot\nabla_{T_\nabla(e_i,e_j)}\varphi,
\end{align*}
where $F(\varphi)=\sum_{i,j=1}^{2n}J\bar e_i\cdot\bar  e_j\cdot R^S(e_i,e_j)\varphi$.
\end{thm}
\begin{proof}
From Lemma 5.4, we have
\begin{align*}
D_{\rm tr}\tilde D_{\rm tr}\varphi=&\sum_{i=1}^{2n} \{\bar e_i\cdot D_{\rm tr}(\nabla_{e_i}\varphi)+P(\bar e_i)\cdot\nabla_{e_i}\varphi-\sqrt{-1}\nabla_{e_i}\nabla_{e_i}\varphi\}\\
&+\sqrt{-1}\nabla_{J(\bar\kappa^\sharp+\tau_\nabla)}\varphi-\frac12 J(\bar\kappa^\sharp+\tau_\nabla)\cdot D_{\rm tr}\varphi-\frac12P(J(\bar\kappa^\sharp+\tau_\nabla))\cdot\varphi-{\sqrt{-1}\over 4}|\bar\kappa^\sharp+\tau_\nabla|^2\varphi.
\end{align*}
Since
\begin{align}\label{6-8}
\sum_{i=1}^{2n} \bar e_i \cdot D_{\rm tr}(\nabla_{e_i}\varphi)=&-\sum_{i,j=1}^{2n}\bar e_i \cdot J\bar e_j \cdot\nabla_{e_j}\nabla_{e_i}\varphi-\frac12(\bar\kappa^\sharp+\tau_\nabla)\cdot \tilde D_{\rm tr}\varphi\\
&-\frac{\sqrt{-1}}2\nabla_{J(\bar\kappa^\sharp+\tau_\nabla)}\varphi-\frac14(\bar\kappa^\sharp+\tau_\nabla)\cdot J(\bar\kappa^\sharp+\tau_\nabla)\cdot\varphi,\notag
\end{align}
we have 
\begin{align*}
D_{\rm tr}\tilde D_{\rm tr}\varphi=&-\sqrt{-1}\sum_{i=1}^{2n} \nabla_{e_i}\nabla_{e_i}\varphi-\sum_{i,j=1}^{2n}\bar e_i\cdot J\bar e_j\cdot\nabla_{e_j}\nabla_{e_i}\varphi +\sum_{i=1}^{2n}P(\bar e_i)\cdot\nabla_{e_i}\varphi\\
&+{\sqrt{-1}\over 2}\nabla_{J(\kappa^\sharp+\tau_\nabla)}\varphi-\frac12 J(\bar\kappa^\sharp+\tau_\nabla)\cdot D_{\rm tr}\varphi-\frac12 (\bar\kappa^\sharp+\tau_\nabla)\cdot \tilde D_{\rm tr}\varphi\\
&-\frac12 P(J(\bar\kappa^\sharp+\tau_\nabla))\cdot\varphi-{\sqrt{-1}\over 4}|\bar\kappa^\sharp+\tau_\nabla|^2\varphi-\frac14 (\bar\kappa^\sharp+\tau_\nabla)\cdot J(\bar\kappa^\sharp+\tau_\nabla)\cdot\varphi.
\end{align*}
Similarly, we have
\begin{align*}
\tilde D_{\rm tr}D_{\rm tr}\varphi=& \sqrt{-1}\sum_{i=1}^{2n}\nabla_{e_i}\nabla_{e_i}\varphi-\sum_{i,j=1}^{2n}J\bar e_i\cdot \bar e_j\cdot\nabla_{e_j}\nabla_{e_i}\varphi -\sum_{i=1}^{2n}P(J\bar e_i)\cdot\nabla_{e_i}\varphi\\
&-{\sqrt{-1}\over 2}\nabla_{J(\bar\kappa^\sharp+\tau_\nabla)}\varphi-\frac12 J(\bar\kappa^\sharp+\tau_\nabla)\cdot D_{\rm tr}\varphi-\frac12 (\bar\kappa^\sharp+\tau_\nabla)\cdot \tilde D_{\rm tr}\varphi\\
&-\frac12 \tilde P(\bar\kappa^\sharp+\tau_\nabla)\cdot\varphi+{\sqrt{-1}\over 4}|\bar\kappa^\sharp+\tau_\nabla|^2\varphi
-\frac14 J(\bar\kappa^\sharp+\tau_\nabla)\cdot (\bar\kappa^\sharp+\tau_\nabla)\cdot \varphi.
\end{align*}
Therefore, we have
\begin{align}\label{5-6}
[\tilde D_{\rm tr},D_{\rm tr}]\varphi=& \sqrt{-1}\sum_{i=1}^{2n}\nabla_{e_i}\nabla_{e_i}\varphi-\sqrt{-1}\nabla_{J(\bar\kappa^\sharp+\tau_\nabla)}\varphi+\sum_{i,j=1}^{2n}J\bar e_i\cdot\bar e_j\cdot R^S(e_i,e_j)\varphi\notag\\
&+\sum_{i,j=1}^{2n}J\bar e_i\cdot\bar e_j\cdot\nabla_{[e_i,e_j]}\varphi-\sum_{i=1}^{2n}\{P(\bar e_i)+\tilde P(J\bar e_i)\}\cdot\nabla_{e_i}\varphi\notag\\ 
& +\frac12\{P(J(\bar\kappa^\sharp+\tau_\nabla))-\tilde P(\bar\kappa^\sharp+\tau_\nabla)\}\cdot\varphi+{\sqrt{-1}\over 4}|\bar\kappa^\sharp+\tau_\nabla|^2\varphi.
\end{align}
On the other hand, by Lemma 5.5, we have 
\begin{align}\label{5-7}
&\sum_{i=1}^{2n}\{P(\bar e_i)+\tilde P(J\bar e_i)\}\cdot\nabla_{e_i}\varphi -\sum_{i,j=1}^{2n}J\bar e_i\cdot \bar e_j\cdot\nabla_{[e_i,e_j]}\varphi  \notag\\
&=-\sum_{i=1}^{2n} P(J)(J\bar e_i)\cdot\nabla_{e_i}\varphi-\sqrt{-1}\sum_{i=1}^{2n} {\rm div}_\nabla(\bar e_i)\nabla_{e_i}\varphi-\sum_{i,j=1}^{2n}\bar e_i\cdot J\bar e_j\cdot\nabla_{T_\nabla(e_i,e_j)}\varphi.
\end{align}
From  Proposition 5.3, (\ref{5-6}) and  (\ref{5-7}), the proof follows. 
\end{proof}

\begin{cor}  Let  $\nabla$ be a transverse Fedosov connection on $(M,\mathcal F,\omega)$ with a transversely metaplectic structure.  Then for any $\varphi\in\Gamma Sp(\mathcal F)$,
\begin{align}\label{5-8}
\mathcal P_{\rm tr}\varphi=& \nabla_{\rm tr}^*\nabla_{\rm tr}\varphi+\sqrt{-1}F(\varphi)-\frac14|\bar\kappa^\sharp|^2\varphi+{\sqrt{-1}\over 2}\{P(J\bar\kappa^\sharp)-\tilde P(\bar\kappa^\sharp)\}\cdot\varphi\\
&+\sqrt{-1}\sum_{i=1}^{2n}P(J)(J\bar e_i)\cdot\nabla_{e_i}\varphi.\notag 
\end{align}
 In addition, if $\nabla J=0$, then
 \begin{align}\label{5-9}
 \mathcal P_{\rm tr}\varphi&= \nabla_{\rm tr}^*\nabla_{\rm tr}\varphi+\sqrt{-1}F(\varphi)-\frac14|\bar\kappa^\sharp|^2\varphi+{\sqrt{-1}\over 2}\{P(J\bar\kappa^\sharp)-\tilde P(\bar\kappa^\sharp)\}\cdot\varphi.
 \end{align}
 \end{cor}
 \begin{proof}  Since $T_\nabla =0$ for the transverse Fedosov connection, the proof of (\ref{5-8}) is trivial.  If $\nabla J=0$, then $P(J)=0$, which proves (\ref{5-9}).
 \end{proof}
 Since the contact flow and cosymplectic flow is minimal (that is, $\kappa=0$) \cite{Pa}, we have the following.
\begin{cor}  Let  $\nabla$ be a transversely symplectic connection on $(M,\mathcal F_\xi,\omega)$ with a transversely metaplectic structure. Then, for any $\varphi \in \Gamma Sp(\mathcal F)$
\begin{align*}
\mathcal P_{\rm tr}\varphi=& \nabla_{\rm tr}^*\nabla_{\rm tr}\varphi+\sqrt{-1}F(\varphi)-\frac14|\tau_\nabla|^2\varphi+{\sqrt{-1}\over 2}\{P(J\tau_\nabla)-\tilde P(\tau_\nabla)\}\cdot\varphi\\
&+\sqrt{-1}\sum_{i=1}^{2n}P(J)(J\bar e_i)\cdot\nabla_{e_i}\varphi+\sqrt{-1}\sum_{i,j=1}^{2n} e_i\cdot J \bar e_j\cdot\nabla_{T_\nabla(e_i,e_j)}\varphi.
\end{align*}
In addition,  if  $\nabla$ is a transverse Fedosove connection such that $\nabla J=0$,  then  for any $\varphi \in \Gamma Sp(\mathcal F)$
\begin{align*}
\mathcal P_{\rm tr}\varphi&= \nabla_{\rm tr}^*\nabla_{\rm tr}\varphi+\sqrt{-1}F(\varphi).
\end{align*}
\end{cor}
 
 \section{Properties on the foliated symplectic spinor bundle}
 Let $(M,\mathcal F,\omega,\nabla)$ be a transversely symplectic foliation with a fixed transversely metaplectic structure and  a transversely symplectic connection.
 First, we recall the properties of the Hermite functions on $\mathbb R^n$. For precise definition,  see \cite{FO, Ha3}.
 
  Let  $H_0:L^2(\mathbb R^n)\to L^2(\mathbb R^n)$ be the Hamilton operator, which is defined by
 \begin{equation}\label{6-1}
 (H_0f)(x) = \frac12\sum_{j=1}^n \Big ( {\partial^2 f\over \partial x_j^2}(x) -x_j^2 f(x)\Big).
 \end{equation}
 Equivalently,  from (\ref{3-1}) we get
 \begin{equation}\label{6-1-1}
 H_0 f = \frac12\sum_{j=1}^n \Big(\sigma(a_j)\sigma(a_j) + \sigma(b_j)\sigma( b_j )\Big)f.
 \end{equation}
  Now, we define the Hermite function $h_\beta\in L^2(\mathbb R^n)$ on $\mathbb R^n$ by
 \begin{equation}\label{6-2}
 h_\beta(x) = h_{\beta_1}(x_1)\cdots h_{\beta_n}(x_n), \quad x=(x_1,\cdots,x_n),
 \end{equation}
where $\beta = (\beta_1,\cdots,\beta_n)$,  $\beta_j (j=1,\cdots,n)$ are nonnegative integers and
 \begin{equation}\label{6-3}
 h_{\ell}(t) = e^{t^2\over 2} {d^\ell\over dt^\ell}(e^{-t^2}),\quad t\in\mathbb R
 \end{equation}
 is the classical Hermite functions on $\mathbb R$. Then  the Hermite functions  form a complete orthogonal system in $L^2(\mathbb R^n)$ of eigenfunctions of $H_0$ \cite{FO}.  In particular,
 \begin{equation}\label{6-4}
 H_0 h_\beta = -(|\beta| + \frac{n}2) h_\beta,
 \end{equation}
 where $|\beta| = \beta_1 + \cdots +\beta_n$.  Let $\mathcal M_\ell $ denote the eigenspace of $H_0$ with eigenvalue $-(\ell + \frac{n}2)$, that is,
 \begin{equation}\label{6-5-1}
 \mathcal M_\ell =\{ f\in L^2(\mathbb R^n)\ |\ H_0 f = -(\ell + {n\over 2})f\}.
 \end{equation}
  Then  by combinatorial  computation, we get
  \begin{equation}\label{6-5}
 \dim_{\mathbb C} \mathcal M_\ell = {}_{n+\ell-1}C_\ell.
 \end{equation}
 Moreover, the spaces $\mathcal M_\ell (l=0,1,\cdots)$ form an orthogonal decomposition of $L^2 (\mathbb R^n)$.

  Let $P_{Sp}^J(Q)$ denote the corresponding $U(n)$-reduction of the symplectic frame bundle $P_{Sp}(Q)$. So the fiber of $P_{Sp}^J(Q)$ at $x\in M$ is the set  of all unitary basis of $Q_x$.
Set 
\begin{equation*}
\tilde P_{Mp}^J(Q) =\Pi^{-1}(P_{Sp}^J(Q)),
\end{equation*}
where $\Pi:\tilde P_{Mp}(Q)\to P_{Sp}(Q)$ is the bundle morphism. 
 Clearly, $\tilde P_{Mp}^J(Q)$  is a principal $\tilde U(n)$-bundle, where $\tilde U(n)\subset Mp(n,\mathbb R)$ is the double cover of $U(n)\subset Sp(n,\mathbb R)$.  Moreover, the foliated symplectic spinor bundle $Sp(\mathcal F)$ is associated to $\tilde P_{Mp}^J(Q)$ by  the restriction $\mathfrak u = \mathfrak m|_{\tilde U(n)}$, i.e., 
\begin{align*}
Sp(\mathcal F) = \tilde P_{Mp}^J(Q)\times_{\mathfrak u} L^2 (\mathbb R^n).
\end{align*} 
Then the bundle $Sp(\mathcal F)$ is decomposed into finite rank subbundles $Sp_\ell^J(\mathcal F)$, where
\begin{equation}\label{6-6}
Sp_\ell^J(\mathcal F) = \tilde P_{Mp}^J(Q)\times_{\mathfrak u_\ell} \mathcal M_\ell,
\end{equation}
where $\mathfrak u_\ell$ is the restriction of the unitary representation $\mathfrak u$ to the subspace $\mathcal M_\ell$, that is, $\mathfrak u_\ell : \tilde U(n)\to U(\mathcal M_\ell)$ is the irreducible representation. From (\ref{6-5}), we have
\begin{equation}\label{6-7}
{\rm rank}_{\mathbb C} Sp_\ell^J(\mathcal F) = {}_{n+\ell-1}C_\ell.
\end{equation}
   On a foliated symplectic spinor bundle $Sp(\mathcal F)$, we define $\mathcal H^J:Sp(\mathcal F)\to Sp(\mathcal F)$ by
 \begin{equation}\label{6-8}
 \mathcal H^J([p,f]) = [p,H_0 f]
 \end{equation}
 for $p\in \tilde P_{Mp}^J(Q)$ and $f\in  L^2(\mathbb R^n)$.
\
From (\ref{6-4}), (\ref{6-6}) and (\ref{6-8}),  we have the following.
\begin{prop}   For any $\varphi \in \Gamma Sp_\ell^J(\mathcal F)$, it holds 
\begin{equation*}
\mathcal H^J (\varphi) = -(\ell + \frac{n}2) \varphi.
\end{equation*}
\end{prop}
\begin{proof}  Let $\varphi \in \Gamma Sp_\ell^J(\mathcal F)$, that is, $\varphi =[p,f]$ for $f\in \mathcal M_\ell$. Then $\mathcal H^J(\varphi) =[p, H_0 f] = -(\ell +\frac{n}2) [p,f] =-(\ell+\frac{n}2)\varphi$.
\end{proof}
\begin{lem}   For any $\varphi \in \Gamma Sp(\mathcal F)$, 
\begin{equation}\label{6-9}
\mathcal H^J(\varphi) = \frac12\sum_{j=1}^{2n} \bar e_j\cdot \bar e_j \cdot \varphi.
\end{equation}
Moreover,  for any $\varphi,\psi\in \Gamma Sp(\mathcal F)$, we have
\begin{equation}\label{6-10}
<\mathcal H^J(\varphi),\psi> = < \varphi,\mathcal H^J(\psi)>.
\end{equation} 
\end{lem}
\begin{proof}  The proof of (\ref{6-9}) is similar to Lemma 3.3.2 in \cite{Ha3}.  That is,  let $\varphi =[p,f]$. Then
\begin{align*}
\mathcal H^J(\varphi)&=[p,H_0 f]\\
&=\frac12\sum_{j=1}^n [p,\sigma(a_j)\sigma(a_j) f + \sigma(b_j)\sigma(b_j) f]\\
&=\frac12\sum_{j=1}^n (\bar e_j\cdot \bar e_j + \bar e_{n+j}\cdot\bar e_{n+j}) \cdot\varphi\\
&=\frac12\sum_{j=1}^{2n}\bar e_j\cdot\bar e_j \cdot \varphi,
\end{align*}
where $\bar e_j\cdot\varphi = [p,\sigma(a_j)f]$ and $\bar e_{n+j} \cdot\varphi = [p,\sigma(b_j)f]$.
  The proof of (\ref{6-10}) follows from (\ref{3-6}) and (\ref{6-9}).
\end{proof}

\begin{prop}  For any  $s\in \Gamma Q$ and $\varphi\in \Gamma Sp(\mathcal F)$, we have
\begin{align}\label{6-13}
\mathcal H^J(s\cdot\varphi) = s\cdot \mathcal H^J(\varphi) + \sqrt{-1} Js\cdot\varphi.
\end{align}
\end{prop}
\begin{proof}  From (\ref{3-4}) and Lemma 6.2,  the proof follows.
\end{proof}
\begin{prop}   For   any  $X\in \Gamma TM$ and  $\varphi\in \Gamma Sp(\mathcal F)$, we have
\begin{align}\label{6-14}
\nabla_X (\mathcal H^J\varphi) = \mathcal H^J (\nabla_X\varphi)+\sum_{j=1}^{2n}J(\nabla_XJ)\bar e_j\cdot \bar e_j\cdot\varphi
\end{align}

In particular, if $\nabla J=0$, then
\begin{equation}\label{6-15}
\nabla_X (\mathcal H^J\varphi) = \mathcal H^J (\nabla_X\varphi).
\end{equation}
\end{prop}
\begin{proof}   From (\ref{6-9}), we have that for any $\varphi$
\begin{equation}\label{6-16}
\nabla_X (\mathcal H^J \varphi)= \frac12\sum_{j=1}^{2n} \{\nabla_X \bar e_j \cdot \bar e_j  + \bar e_j\cdot\nabla_X \bar e_j \} \cdot\varphi + \mathcal H^J(\nabla_X\varphi).
\end{equation}
Since $\nabla$ is transversely symplectic, we get
\begin{equation}\label{6-17}
\omega_Q((\nabla_XJ)\bar e_i,\bar e_j)= \omega_Q(\nabla_X \bar e_j,J\bar e_i) + \omega_Q(\nabla_X \bar e_j,J\bar e_i).
\end{equation}
Note that for any $s\in \Gamma Q$,  $s=\sum_{j=1}^{2n}\omega_Q(s,J\bar e_j)\bar e_j$.  Then   from  (\ref{6-17}), we get 
\begin{align*}
\sum_{j=1}^{2n} \{\nabla_X \bar e_j \cdot \bar e_j  + \bar e_j\cdot\nabla_X \bar e_j \} \cdot\varphi &=\sum_{i,j=1}^{2n} \{\omega_Q(\nabla_X \bar e_i,J\bar e_j)\bar e_j\cdot \bar e_i + \omega_Q(\nabla_X \bar e_i,J\bar e_j) \bar e_i\cdot \bar e_j\}\cdot\varphi\\
&=\sum_{i,j=1}^{2n} \{\omega_Q(\nabla_X \bar e_i,J\bar e_j) + \omega_Q(\nabla_X \bar e_j,J\bar e_i)\}\bar e_j\cdot\bar e_i\cdot\varphi\\&=\sum_{i,j=1}^{2n}\omega_Q((\nabla_XJ)\bar e_i,\bar e_j)\bar e_j\cdot\bar e_i\cdot\varphi.
\end{align*}
 From (\ref{6-16}),  the proof of (\ref{6-14}) follows.  The proof of (\ref{6-15}) is trivial from (\ref{6-14}).
\end{proof}

\begin{cor}  If  $\nabla J=0$,    then  for any $X\in \Gamma TM$ and $\varphi\in\Gamma Sp_\ell^J(\mathcal F)$,
\begin{equation*}
\nabla_X \varphi \in \Gamma Sp_\ell^J(\mathcal F).
\end{equation*}
\end{cor}

\section{Vanishing of the special spinors in $Sp_o^J(\mathcal F)$}
Let $(M,\mathcal F,\omega,\nabla)$ be a transversely symplectic foliation with a fixed  transversely metaplectic structure and a transversely symplectic connection. 
\begin{prop}  If   $\nabla J=0$, then for any $\varphi \in \Gamma Sp(\mathcal F)$, 
\begin{align}
\mathcal H^J(D_{\rm tr}\varphi) &= D_{\rm tr}(\mathcal H^J\varphi) + \sqrt{-1} \tilde D_{\rm tr}\varphi\label{7-1}\\
\mathcal H^J(\tilde D_{\rm tr}\varphi)&= \tilde D_{\rm tr}(\mathcal H^J\varphi) - \sqrt{-1} D_{\rm tr}\varphi\label{7-2}\\
\mathcal H^J(\mathcal P_{\rm tr}\varphi)& = \mathcal P_{\rm tr}(\mathcal H^J\varphi).\label{7-3-1}
\end{align}
Trivially,  $\mathcal P_{\rm tr}$ preserves the space $\Gamma Sp_\ell^J(\mathcal F)$.
\end{prop}
\begin{proof}  For any $\varphi\in \Gamma Sp(\mathcal F)$, we have from (\ref{4-7}) and (\ref{6-13})
\begin{equation}\label{7-3}
\mathcal H^J(D_{\rm tr}\varphi)= \mathcal H^J(D_{\rm tr}' \varphi) -\frac12(\bar\kappa^\sharp+\tau_\nabla)\cdot \mathcal H^J (\varphi)-\frac{\sqrt{-1}}2 J(\bar\kappa^\sharp+\tau_\nabla)\cdot\varphi.
\end{equation}
From (\ref{6-13}) and (\ref{6-15}), we have
\begin{equation}\label{7-4}
\mathcal H^J(D_{\rm tr}' \varphi) = D_{\rm tr}'(\mathcal H^J \varphi) + \sqrt{-1} \tilde D_{\rm tr}'\varphi.
\end{equation}
Hence  from   (\ref{7-3}) and (\ref{7-4})
\begin{align*}
\mathcal H^J(D_{\rm tr}\varphi)=  (D_{\rm tr}' - \frac12(\bar\kappa^\sharp +\tau_\nabla))\mathcal H^J\varphi  + \sqrt{-1}(\tilde D_{\rm tr}' -\frac12 J(\bar\kappa^\sharp +\tau_\nabla))\cdot\varphi,
\end{align*}
which proves (\ref{7-1}).  The proof of (\ref{7-2}) is similary proved.  The proof of (\ref{7-3-1}) follows from (\ref{7-1}) and (\ref{7-2}).  The last statement is proved from Proposition 6.1.
\end{proof}
Let $\mathcal P^\ell_{\rm tr} = \mathcal P_{\rm tr}|_{Sp_\ell^J(\mathcal F)}$. 
In what follows, we study the Weitzenb\"ock  formula for $\mathcal P_{\rm tr}^0$ on $Sp_0^J(\mathcal F)$. 
\begin{prop}   For  any $s\in   \Gamma Q$ and $\varphi\in \Gamma Sp_0^J(\mathcal F)$,  we get
\begin{equation*}
Js\cdot\varphi = \sqrt{-1}s\cdot\varphi.
\end{equation*}
\end{prop}
\begin{proof} The proof is similar to Corollary 3.3.7 in \cite{Ha3}. Let $s=[p, v]\in  \Gamma Q=\tilde P_{Mp}^J(Q)\times_\rho \mathbb R^{2n}$. Then $Js=[ p, J_0 v]$, where  $J_0(v_1,v_2)=(-v_2,v_1)$ for $v=(v_1,v_2)\in\mathbb R^{2n}$,  $v_j\in \mathbb R^n$. Now let $\varphi=[p,f]\in Sp_0^J(\mathcal F)$, where $H_0 f =-{n\over 2}f$ for  $f\in L^2(\mathbb R^n)$.  That is, $f$ satisfies 
\begin{equation}\label{7-5}
\sum_{j=1}^n \Big( {\partial ^2 f\over \partial x_j^2 }-x_j^2 f + f\Big)=0.
\end{equation}
On the other hand,  a function $f$ satisfying 
\begin{equation}\label{7-6}
{\partial f\over\partial x_j} = -x_j f
\end{equation} 
is a solution of (\ref{7-5}).  Since the rank of $Sp_0^J(\mathcal F)$ is one,  a solution of (\ref{7-5}) is also the one of  (\ref{7-6}). 
Hence  (\ref{7-6}) yields   $\sigma(b_j)f = \sqrt{-1}\sigma(a_j) f$.  Since $J_0 a_j =b_j$ and $J_0 b_j =-a_j$, we have that $ \sigma(J_0 a_j )= \sqrt{-1}\sigma(a_j)$ and $\sigma(J_0 b_j)=\sqrt{-1}\sigma(b_j)$, and so $\sigma(J_0v) = \sqrt{-1}\sigma(v)$ for any $v$.  Hence from (\ref{3-2}),
\begin{equation*}
Js\cdot \varphi = [p, \sigma(J_0v) f] = [p, \sqrt{-1}\sigma(v) f] =\sqrt{-1}s\cdot\varphi,
\end{equation*}
which finishes the proof.
\end{proof}
\begin{lem} If $\nabla J=0$, then  for any $s\in   \Gamma Q$  and $\varphi\in \Gamma Sp^J_0(\mathcal F)$, 
\begin{align*}
\{P(Js)-\tilde P(s)\}\cdot\varphi ={\rm div}_\nabla(s^c)\varphi,
\end{align*}
where $s^c=s -\sqrt{-1}Js$.
\end{lem}
\begin{proof}  Let $\varphi\in \Gamma Sp^J_0(\mathcal F)$.  From  $\nabla J=0$ and Proposition 7.2, we have
\begin{align*}
P(Js)\cdot\varphi&=\sum_{j=1}^{2n}\bar e_j\cdot \nabla_{Je_j}Js\cdot\varphi
=\sum_{j=1}^{2n} \bar e_j\cdot  J(\nabla_{Je_j}s)\cdot\varphi\\
&=\sum_{j=1}^{2n} \sqrt{-1}\bar e_j\cdot \nabla_{Je_j}s\cdot\varphi
=-\sqrt{-1}\sum_j J\bar e_j \cdot \nabla_{e_j}s\cdot\varphi\\
&=-\sqrt{-1}\sum_{j=1}^{2n} \nabla_{e_j}s\cdot J\bar e_j \cdot\varphi + \sum_{j=1}^{2n}\omega_Q(\nabla_{e_j}s,J\bar e_j)\varphi\\
&=\sum_{j=1}^{2n} \nabla_{e_j}s\cdot \bar e_j \cdot\varphi + \sum_{j=1}^{2n}\omega_Q(\nabla_{e_j}s,J\bar e_j)\varphi\\
&=\sum_{j=1}^{2n} \{\bar e_j\cdot \nabla_{e_j}s -\sqrt{-1}\omega_Q(\nabla_{e_j}s,J\bar e_j) +\omega_Q(\nabla_{e_j}s,J\bar e_j)\}\cdot\varphi\\
&=\tilde P(s)\cdot\varphi + {\rm div}_\nabla(s)\varphi -\sqrt{-1}{\rm div}_\nabla(Js)\varphi\\
&=\tilde P(s)\cdot\varphi + {\rm div}_\nabla(s-\sqrt{-1}Js)\varphi,
\end{align*}
which yields the proof.
\end{proof}
 Let ${\rm Sric}^\nabla$ and $r^\nabla$ be the {\it transversal symplectic Ricci tensor} and {\it transversal symplectic scalar curvature} of $\nabla$ on $\mathcal F$, which are defined by
\begin{align}
&{\rm Sric}^\nabla(s,t)=\sum_{j=1}^n\omega_Q(R^\nabla(v_j,w_j)s,t),\label{7-7}\\
&r^\nabla = \sum_{j=1}^{2n}{\rm Sric}^\nabla(\bar e_j,\bar e_j)=\frac12\sum_{i,j=1}^{2n}\omega_Q(R^\nabla(e_i,Je_i)\bar e_j,\bar e_j),\label{7-7-1}
\end{align}
respectively, where $\{e_j\}_{j=1,\cdots,2n}$ is a unitary basic  frame of $\mathcal F$.  

Note that if $\nabla$ is transverse Fedosov (that is, symplectic and torsion-free), then ${\rm Ric}^\nabla = {\rm Sric}^\nabla$ \cite{Ha3}, where ${\rm Ric}^\nabla$ is the transversal Ricci tensor on $Q$, that is, ${\rm Ric}^\nabla(X) = \sum_{j=1}^{2n} R^\nabla(X,e_j)e_j$ for any normal vector field $X\in\Gamma Q$.

\begin{prop} If $\nabla J=0$, then for any $\varphi \in \Gamma Sp_0^J(\mathcal F)$,
\begin{align}\label{7-8}
F(\varphi) =
 {\sqrt{-1}\over 4} r^\nabla \varphi.
\end{align}
\end{prop}
\begin{proof}  Let $\{e_1,\cdots,e_{2n}\}$ be a unitary basic  frame and   $\varphi \in \Gamma Sp_0^J(\mathcal F)$.  Since $\nabla J=0$, by Corollary 6.5,  $R^S(e_i,e_j)\varphi \in\Gamma Sp_0^J(\mathcal F)$.
Since $\sum_{i,j=1}^{2n}\omega_Q(J\bar e_i,\bar e_j)R^S(e_i,e_j)\varphi =0$,  from Proposition 7.2, we have
\begin{align}\label{7-9}
F(\varphi)&=\sum_{i,j=1}^{2n}J\bar e_i\cdot\bar e_j\cdot R^S(e_i,e_j)\varphi\notag\\
&= \sum_{i,j=1}^{2n} \bar e_j\cdot J\bar e_i\cdot R^S(e_i,e_j)\varphi \notag\\
&=\sqrt{-1}\sum_{i,j=1}^{2n} \bar e_j\cdot \bar e_i \cdot R^S(e_i,e_j)\varphi\notag\\
&={1\over 2}\sum_{i,j=1}^{2n}\omega_Q(\bar e_i,\bar e_j)R^S(e_j,e_i)\varphi\notag\\
&=\frac12\sum_{i=1}^{2n} R^S(Je_i,e_i)\varphi.
\end{align}
Since  $[R^\nabla(X,Y),J]=0$ for any $X,Y\in \Gamma TM$,  from (\ref{3-8}) and Proposition 7.2,
\begin{align}\label{7-10}
R^S(X,Y)\varphi &= {\sqrt{-1}\over 2}\sum_{i=1}^{2n} \bar e_i \cdot R^\nabla (X,Y) J\bar e_i \cdot\varphi\notag\\
&={\sqrt{-1}\over 2} \sum_{i=1}^{2n} \bar e_i\cdot JR^\nabla(X,Y)\bar e_i \cdot\varphi\notag\\
&=-\frac12 \sum_{i=1}^{2n} \bar e_i\cdot R^\nabla(X,Y)\bar e_i\cdot\varphi\notag\\
&= \frac12 \sum_{i,j=1}^{2n} \omega_Q(R^\nabla(X,Y) J\bar e_i,\bar e_j)\bar e_i\cdot\bar e_j \cdot\varphi.
\end{align}
From Lemma 3.1,  we get
\begin{equation*}
\sum_{i,j=1}^{2n}\omega_Q(R^\nabla(X,Y)J\bar e_i,\bar e_j)\bar e_i\cdot\bar e_j\cdot \varphi = -\sum_{i,j=1}^{2n}\omega_Q(R^\nabla(X,Y)J\bar e_i,\bar e_j)\bar e_j\cdot \bar e_i\cdot\varphi,
\end{equation*}
 which implies
\begin{align}\label{7-11}
\sum_{i,j=1}^{2n} \omega_Q(R^\nabla(X,Y)J\bar e_i,\bar e_j)\bar e_i\cdot \bar e_j \cdot\varphi&= \frac12\sum_{i,j=1}^{2n}\omega_Q(R^\nabla(X,Y)J\bar e_i,\bar e_j)(\bar e_i\cdot \bar e_j - \bar e_j\cdot \bar e_i)\cdot\varphi\notag\\
&=-{\sqrt{-1}\over 2} \sum_{i,j=1}^{2n}\omega_Q(R^\nabla(X,Y)J\bar e_i,\bar e_j)\omega_Q(\bar e_i,\bar e_j)\varphi\notag\\
&= -{\sqrt{-1}\over 2} \sum_{i=1}^{2n}\omega_Q(R^\nabla(X,Y)\bar e_i,\bar e_i)\varphi.
\end{align}
From (\ref{7-10}) and (\ref{7-11}), we have
\begin{equation*}
R^S(X,Y)\varphi = -{\sqrt{-1}\over 4} \sum_{i=1}^{2n} \omega_Q(R^\nabla(X,Y)\bar e_i,\bar e_i)\varphi,
\end{equation*}
which implies 
\begin{equation}\label{7-12}
\sum_{i=1}^{2n}R^S(Je_i,e_i)\varphi = {\sqrt{-1}\over 4}\sum_{i,j=1}^{2n} \omega_Q(R^\nabla(e_i,Je_i)\bar e_j,\bar e_j)\varphi={\sqrt{-1}\over 2}r^\nabla \varphi.
\end{equation}
Hence (\ref{7-8}) follows from (\ref{7-9}) and (\ref{7-12}).
\end{proof}
\begin{thm}    If $\nabla J=0$, then for any $\varphi \in \Gamma Sp_0 ^J(\mathcal F)$,  we have
\begin{align*}
\mathcal P^0_{\rm tr}\varphi  = \nabla_{\rm tr}^*\nabla_{\rm tr}\varphi-{1\over 4}(r^\nabla +|\tau_\nabla+\bar\kappa^\sharp|^2)\varphi
+{\sqrt{-1}\over 2}{\rm div}_\nabla(\tau_\nabla+\bar\kappa^\sharp)^c\cdot\varphi+\sqrt{-1}\nabla_{\tau_\nabla}\varphi.
\end{align*}
\end{thm}
\begin{proof}  Let  $\varphi\in \Gamma Sp_0^J(\mathcal F)$. Then from Corollary 6.5 and Proposition 7.2, we get
\begin{align*}
\sum_{i,j=1}^{2n} \bar e_i\cdot J\bar e_j \cdot\nabla_{T_\nabla(e_i,e_j)}\varphi 
&= \sqrt{-1}\sum_{i,j=1}^{2n}\bar e_i\cdot \bar e_j \cdot \nabla_{T_\nabla(e_i,e_j)}\varphi \\
&=\frac12\sum_{i,j=1}^{2n}\omega_Q(\bar e_i,\bar e_j) \nabla_{T_\nabla(e_i,e_j)}\varphi\\
&=\frac12 \sum_{i=1}^{2n} \nabla_{T_\nabla (e_i,Je_i)}\varphi\\
&=\nabla_{\tau_\nabla}\varphi.
\end{align*}
Since $\sum_{i=1}^{2n} T_\nabla(e_i,Je_i)  = 2\tau_\nabla$, the last equality in the above holds. From Theorem 5.6, Lemma 7.3  and Proposition 7.4, the proof is completed.
\end{proof}

\begin{cor} Let  $\nabla$ be a transverse Fedosov connection on $(M,\mathcal F,\omega)$ with a transversely metaplectic structure. If $\nabla J=0$, then for any $\varphi\in \Gamma Sp_0^J(\mathcal F)$
\begin{align}\label{7-13}
\mathcal P_{\rm tr}^0\phi & = \nabla_{\rm tr}^*\nabla_{\rm tr}\varphi-{1\over 4}(r^\nabla +|\bar\kappa^\sharp|^2)\varphi+{\sqrt{-1}\over 2}{\rm div}_\nabla(\bar\kappa^\sharp)^c\cdot\varphi.
\end{align}
In addition, if $\mathcal F$ is minimal, then
\begin{equation}\label{7-14}
\mathcal P_{\rm tr}\varphi= \nabla_{\rm tr}^*\nabla_{\rm tr}\varphi - \frac14 r^\nabla\varphi.
\end{equation}
\end{cor}
\begin{proof} Since $\tau_\nabla=0$ and $P(J)=0$,  the proof follows from Theorem 7.5.
 \end{proof}
\begin{cor}  Let  $\nabla$ be a transversely symplectic connection such that $\nabla J=0$ on $(M,\mathcal F_\xi,\omega)$  with a  transversely metaplectic structure.  Then for any $\varphi \in \Gamma Sp_0^J(\mathcal F)$
\begin{align*}
\mathcal P_{\rm tr}^0\phi & = \nabla_{\rm tr}^*\nabla_{\rm tr}\varphi-{1\over 4}(r^\nabla +|\tau_\nabla|^2)\varphi+{\sqrt{-1}\over 2}{\rm div}_\nabla(\tau_\nabla^c)\varphi+\sqrt{-1}\nabla_{\tau_\nabla}\varphi.
\end{align*}
In addition, if $\nabla$ is transverse Fedosov,  then
\begin{align*}
\mathcal P_{\rm tr}^0\varphi&= \nabla_{\rm tr}^*\nabla_{\rm tr}\varphi - \frac14 r^\nabla\varphi.
\end{align*}
 \end{cor}

\begin{thm}  Let $\nabla$ be a transverse Fedosov connection such that $\nabla J=0$ on a closed  $(M,\mathcal F,\omega)$ with a transversely metaplectic structure.  If $\mathcal F$ is minimal and the transversal symplectic scalar curvature is negative, then $\ker \mathcal P_{\rm tr}^0 =\{0\}$. 
\end{thm}
\begin{proof}  Let $\varphi\in\ker \mathcal P_{\rm tr}^0$. From  (\ref{7-14}), by integrating 
\begin{equation*}
\int_M |\nabla_{\rm tr}\varphi|^2 -\frac14\int_M r^\nabla |\varphi|^2 =0.
  \end{equation*}
 By the assumption of the symplectic scalar curvature, we have $\varphi=0$. 
 \end{proof}
\section{Transversely symplectic Dirac operators on transverse K\"ahler foliations}

Let $(M,\mathcal F,J,g_Q)$ be a transverse K\"ahler foliation with  a holonomy invariant transverse complex structure $J$ and transverse Hermitian metric $g_Q$ on $M$. Let $\omega_Q$ be a basic K\"ahler 2-form associated to $g_Q$. It is well known that a transverse K\"ahler foliation is a transversely symplectic foliation with the transverse symplectic form $\omega_Q$. Trivially, the transverse Levi-Civita connection $\nabla$ is the transversely Fedosov connection with $\nabla J=0$. That is,  a transverse K\"ahler foliation is transverse Fedosov.  
Throughout this section, we fix the transverse Levi-Civita connection and  a transversely metaplectic structure $\tilde P_{Mp} (Q)$.  Let $\{e_j\} (j=1,\cdots,2n)$ be a local orthonormal basic frame on $Q=T\mathcal F^\perp$.

\begin{prop}  Let $(M,\mathcal F,J,g_Q)$ be a transverse K\"ahler foliation on a closed manifold $M$. Then the operator $\mathcal P_{\rm tr}$ is  formally self-adjoint.
\end{prop}
\begin{proof} Since the transverse Levi-Civita connection $\nabla$ satisfies $\nabla J=0$, 
 from Theorem 4.2, $D_{\rm tr}$ and $\tilde D_{\rm tr}$ are formally self-adjoint.  So $\mathcal P_{\rm tr}=\sqrt{-1}[\tilde D_{\rm tr},D_{\rm tr}]$ is formally self-adjoint.
\end{proof}
\begin{lem}  On a transverse K\"ahler foliation, we get
\begin{equation}\label{8-1}
P(J\kappa^\sharp)-\tilde P(\kappa^\sharp) =P(\kappa^{\sharp_g})+\tilde P(J\kappa^{\sharp_g}).
\end{equation}
In particular, if the mean curvature vector $\kappa$ of $\mathcal F$ is automorphic, i.e., $J\nabla_Y \kappa^{\sharp_g} = \nabla_{JY}\kappa^{\sharp_g}$ for any $Y\in\mathfrak X_B(\mathcal F)$,  then
\begin{align}\label{8-2}
2P(\kappa^{\sharp_g})=2\tilde P(J\kappa^{\sharp_g}),
\end{align}
where $\kappa(X) = g_Q(\kappa^{\sharp_g},X)$ for any $X\in \Gamma Q$.
\end{lem}
\begin{proof}   Since $\kappa = i (\kappa^\sharp)\omega=g_Q(J\kappa^\sharp,\cdot)$, we know that  $J\kappa^\sharp = \kappa^{\sharp_g}$.  Hence the proof of (\ref{8-1}) is proved.
On the other hand, the condition of $\kappa$ yields
\begin{align*}
\tilde P(J\kappa^{\sharp_g}) &= \sum_{j=1}^{2n} e_j \cdot\nabla_{e_j}J\kappa^{\sharp_g} = \sum_{j=1}^{2n}e_j\cdot J\nabla_{e_j}\kappa^{\sharp_g} = \sum_{j=1}^{2n}e_j\cdot\nabla_{Je_j}\kappa^{\sharp_g} = P(\kappa^{\sharp_g}),
\end{align*}
which proves (\ref{8-2}).
\end{proof}

\begin{thm} On a transverse K\"ahler foliation,  the following holds:
for any $\varphi \in \Gamma Sp(\mathcal F)$
\begin{equation}\label{8-3}
\mathcal P_{\rm tr}\varphi= \nabla_{\rm tr}^*\nabla_{\rm tr}\varphi+\sqrt{-1}F(\varphi)-\frac14|\kappa|^2\varphi+{\sqrt{-1}\over 2}\{P(\kappa^{\sharp_g})+\tilde P(J\kappa^{\sharp_g})\}\cdot\varphi.
\end{equation}
In particular, for any $\varphi\in Sp_0^J(\mathcal F)$
\begin{equation}\label{8-4}
\mathcal P_{\rm tr}^0\varphi= \nabla_{\rm tr}^*\nabla_{\rm tr}\varphi-\frac14(r^\nabla +|\kappa|^2)\varphi+{1\over 2}{\rm div}_\nabla((\kappa^{\sharp_g})^c)\varphi,
\end{equation}
where $r^\nabla$ is the transversal symplectic scalar curvature of $\nabla$.
\end{thm}
\begin{proof}  Since  the transversal Levi-Civita connection $\nabla$ satisfies $\nabla J=0$, the proof of (\ref{8-3})  follow from (\ref{5-9}) and (\ref{8-1}). Since 
\begin{align*}
(\kappa^\sharp)^c = -\sqrt{-1}(\kappa^{\sharp_g})^c,
\end{align*} 
the proof of  (\ref{8-4}) follows  from Corollary 7.6.  
\end{proof}
\begin{cor} If  a transverse K\"ahler foliation  is taut, then  any  spinor field $\varphi\in Sp_0^J(\mathcal F)$ satisfies
\begin{equation*}
\mathcal P_{\rm tr}^0\varphi= \nabla_{\rm tr}^*\nabla_{\rm tr}\varphi-\frac14r^\nabla \varphi.
\end{equation*}
\end{cor}
\begin{proof}  Since $\mathcal F$ is taut,  we can choose a bundle-like metric such that $\kappa^{\sharp_g}=0$.  So the proof  follows from (\ref{8-4}).
\end{proof}
\begin{lem} On a transverse K\"ahler foliation, we have
\begin{equation}\label{8-5}
{\rm Ric}^\nabla(X) =\frac12\sum_{j=1}^{2n} R^\nabla(e_j,Je_j)JX
\end{equation}
for any normal vector field $X\in\Gamma Q\cong T\mathcal F^\perp$.
\end{lem}
\begin{proof}  From Lemma 3.1 and Bianchi's identity, we have that for any $X,Y\in \Gamma Q$,
\begin{align*}
\omega_Q({\rm Ric}^\nabla(X),JY)&= \sum_{j=1}^{2n} \omega_Q(R(X,e_j)e_j,JY)\\
&=-\sum_{j=1}^{2n}\omega_Q(R^\nabla(X,e_j)Je_j,Y)\\
&=\sum_{j=1}^{2n}\omega_Q(R^\nabla(e_j,Je_j)X +R^\nabla(Je_j,X)e_j,Y)\\
&=\sum_{j=1}^{2n}\omega_Q(R^\nabla(e_j,Je_j)X,Y) +\sum_{j=1}^{2n}\omega_Q(R^\nabla(Je_j,X)Je_j,JY)\\
&=\sum_{j=1}^{2n}\omega_Q(R^\nabla(e_j,Je_j)X,Y) -\omega_Q({\rm Ric}^\nabla(X),JY),
\end{align*}
which implies (\ref{8-5}).
\end{proof}
\begin{lem}  On a transverse K\"ahler foliation,   any spinor field $\varphi\in\Gamma Sp(\mathcal F)$ satisfies
\begin{align}
&\sum_{j=1}^{2n}R^S(e_j,Je_j)\varphi =\sqrt{-1} \sum_{j=1}^{2n} {\rm Ric}^\nabla(e_j)\cdot e_j \cdot\varphi,\label{8-6}\\
&\sum_{j=1}^{2n} {\rm Ric}^\nabla(e_j)\cdot Je_j\cdot \varphi =-{\sqrt{-1}\over 2}r^\nabla \varphi.\label{8-7}
\end{align}
\end{lem}
\begin{proof}  From (\ref{3-9}) and Lemma 8.5, we have
\begin{align*}
R^S(e_j,Je_j)\varphi&=- {1\over 2i} \sum_{j,k=1}^{2n} e_k\cdot R^\nabla(e_j,Je_j)Je_k \cdot\varphi\\
&=- {1\over 2i} \sum_{j,k=1}^{2n} R^\nabla(e_j,Je_j)Je_k\cdot e_k  \cdot\varphi\\
&=\sqrt{-1}\sum_{k=1}^{2n} {\rm Ric}^\nabla(e_k)\cdot e_k\cdot\varphi,
\end{align*}
which proves (\ref{8-6}).  For the proof of (\ref{8-7}), we note that from Lemma 3.1
\begin{equation}\label{8-8}
\omega_Q({\rm Ric}^\nabla(X),Y)=\omega_Q(X,{\rm Ric}^\nabla(Y))
\end{equation}
for all normal vector fields $X,Y$.  Hence from (\ref{8-8}) 
\begin{align}\label{8-9}
\sum_{j=1}^{2n} {\rm Ric}^\nabla(e_j)\cdot Je_j\cdot\varphi&=\sum_{j,k=1}^{2n} \omega_Q({\rm Ric}^\nabla(e_j),e_k)Je_k\cdot Je_j\cdot\varphi\notag\\
&=\sum_{j,k=1}^{2n}\omega_Q(e_j,{\rm Ric}^\nabla(e_k))Je_k\cdot Je_j\cdot\varphi\notag\\
&=-\sum_{k=1}^{2n} Je_k\cdot {\rm Ric}^\nabla(e_k)\cdot\varphi.
\end{align}
From (\ref{7-7}) and (\ref{8-9}), we have
\begin{align*}
\sqrt{-1} r^\nabla\varphi&= \sqrt{-1}\sum_{j=1}^{2n} \omega_Q({\rm Ric}^\nabla(e_j),Je_j)\varphi\\
&=\sum_{j=1}^{2n} \{{\rm Ric}^\nabla(e_j)\cdot Je_j - Je_j\cdot {\rm Ric}^\nabla(e_j)\}\varphi\\
&=2\sum_{j=1}^{2n} {\rm Ric}^\nabla(e_j)\cdot Je_j\cdot\varphi,
\end{align*}
which proves (\ref{8-7}).
\end{proof}
\begin{lem}  On a transverse K\"ahler foliation, any spinor field $\varphi\in\Gamma Sp(\mathcal F)$ satisfies
\begin{equation*}
\sum_{j=1}^{2n} Je_j\cdot e_j\cdot\varphi =\sqrt{-1}n\varphi.
\end{equation*}
\end{lem}
\begin{proof}   By a direct calculation, we get
\begin{align*}
\sum_{j=1}^{2n} Je_j \cdot e_j\cdot\varphi = \frac12\sum_{j=1}^{2n} \{Je_j\cdot e_j - e_j\cdot Je_j\}\cdot \varphi = \frac{\sqrt{-1}}2 \sum_{j=1}^{2n} \omega_Q(e_j,Je_j)\varphi =\sqrt{-1}n\varphi.
\end{align*}
\end{proof}
\begin{defn}  Let $h$ be a basic function on $M$.  A  transverse K\"ahler foliation  is said to be  of {\it constant holomorphic sectional curvature} $h$ if
\begin{equation*}
\omega_Q(R^\nabla(X,JX)X,X) = h \omega_Q(X,JX)^2
\end{equation*}
for any normal vector field $X\in T\mathcal F^\perp$.
\end{defn}
\begin{prop} A transverse K\"ahler foliation  is of constant holomorphic sectional curvature $h$ if and only if
\begin{align*}
\omega_Q(R^\nabla(X,Y)Z,W)&={h\over 4} \{\omega_Q(X,Z)\omega_Q(Y,JW) + \omega_Q(X,W)\omega_Q(Y,JZ)-\omega_Q(Y,Z)\omega_Q(X,JW) \\
&-\omega_Q(Y,W)\omega_Q(X,JZ)
+2\omega_Q(X,Y)\omega_Q(Z,JW)\}
\end{align*}
for any normal vector fields $X,Y,Z,W$.
\end{prop}
\begin{proof}  The proof is trivial from \cite{KN}.
\end{proof}

\begin{thm}  On a transverse  K\"ahler foliation of constant holomorphic sectional curvature $h$,  it holds that for any $\varphi\in\Gamma Sp(\mathcal F)$
\begin{align}\label{8-10}
\mathcal P_{tr}\varphi =\nabla_{\rm tr}^*\nabla_{\rm tr}\varphi + \frac{h}4 n(n-1)\varphi - 2h (\mathcal H^J)^2\varphi-\frac14|\kappa|^2\varphi+{\sqrt{-1}\over 2}\{P(\kappa^{\sharp_g})+\tilde P(J\kappa^{\sharp_g})\}\cdot\varphi.
\end{align}
In particular, for any $\varphi\in\Gamma Sp_0^J(\mathcal F)$
\begin{align}\label{8-11}
\mathcal P_{tr}^0\varphi =\nabla_{\rm tr}^*\nabla_{\rm tr}\varphi - \frac{h}4 n(n+1)\varphi -\frac14|\kappa|^2\varphi+{1\over 2}{\rm div}_\nabla(\kappa^{\sharp_g})^{c}\cdot\varphi.
\end{align}
\end{thm}
\begin{proof}   From (\ref{3-9}) and (\ref{3-12}), we get
\begin{align*}
F(\varphi)&=\sum_{i,j=1}^{2n} Je_i \cdot e_j\cdot R^S(e_i,e_j)\varphi \\
&= {\sqrt{-1} \over 2} \sum_{i,j,k,l} \omega_Q(R^\nabla(e_i,e_j)e_k,e_l)Je_i\cdot e_j\cdot e_k\cdot e_l\cdot\varphi
\end{align*}
From Proposition 8.9, we get
\begin{align*}
\omega_Q&(R^\nabla(e_i,e_j)e_k,e_l)\\
& ={h\over 4} \{\delta_{jl}\omega_Q(e_i,e_k) + \delta_{jk}\omega_Q(e_i,e_l)-\delta_{il}\omega_Q(e_j,e_k)-\delta_{ik}\omega_Q(e_j,e_l)+ 2\delta_{kl}\omega_Q(e_i,e_j)\}
\end{align*}
Since $\sum_{j=1}^{2n} \omega_Q(X,e_j)e_j = JX$ for any $X\in T\mathcal F^\perp$,  we have
\begin{align*}
\sqrt{-1}F(\varphi)= &-{h\over 8} \sum_{i,j=1}^{2n} (Je_i\cdot Je_j\cdot e_i\cdot e_j + Je_i\cdot Je_j\cdot e_j\cdot e_i + e_i\cdot e_j \cdot e_j \cdot e_i\\
&+e_i\cdot e_j\cdot e_i\cdot e_j + 2e_i\cdot e_i\cdot e_j\cdot e_j)\varphi\\
=&-{h\over 8} \sum_{i,j=1}^{2n}(Je_i\cdot e_i\cdot Je_j\cdot e_j +Je_i\cdot Je_j\cdot e_j\cdot e_i + 4 e_i\cdot e_i\cdot e_j \cdot e_j)\cdot\varphi \\
&+{\sqrt{-1}h\over 4} \sum_{j=1}^{2n}Je_j\cdot e_j\cdot\varphi. 
\end{align*}	
From Proposition 6.3 and Lemma 8.6, we get
\begin{align*}
&\sum_{i,j} Je_i\cdot Je_j\cdot e_j\cdot e_i\cdot\varphi =-n^2\varphi,\\
&\sum_{i,j}  e_i\cdot e_i\cdot e_j\cdot e_j\cdot\varphi = 4(\mathcal H^J)^2\varphi.
\end{align*}
Hence 
\begin{align*}
\sqrt{-1}F(\varphi) &= -{h\over 8}(-2n^2 \varphi + 16(\mathcal H^J)^2\varphi -{h\over 4}n \varphi\\
 &= {h\over 4}n(n-1) \varphi -2(\mathcal H^J)^2\varphi. 
\end{align*}
The proof of (\ref{8-10}) follows from (\ref{8-3}).   For $\varphi\in\Gamma Sp_0^J(\mathcal F)$,  $\mathcal H^J(\varphi) =-\frac{n}2\varphi$.  Hence the proof of (\ref{8-11}) follows from Lemma 7.3 and (\ref{8-10}).
\end{proof}
\begin{prop} \cite[Proposition 3.10]{JR} Let $(M,\mathcal F,J,g_Q)$ be  a transverse K\"ahler foliation on a closed manifold $M$. Then there exists a bundle-like metric compatible with the K\"ahler structure such that $\kappa$ is basic harmonic; that is, $\delta_B\kappa =\delta_B(J\kappa) =0$ and $\kappa=\kappa_B$, where $\kappa_B$ is the basic part of $\kappa$.
\end{prop}
\begin{lem}  On a transverse K\"ahler foliation, we have
\begin{equation*}
{\rm div}_\nabla(\kappa^{\sharp_g})^c =|\kappa|^2.
\end{equation*}
\end{lem}
\begin{proof}  Let $\delta_T $ be the divergence on the local quotient manifolds in the foliation charts.  That is, $\delta_T= - \sum_{j=1}^{2n} i(e_j)\nabla_{e_j}$ for a local basic frame of $\mathcal F$.  Let $\delta_B$ be the adjoint operator of $d$, that is,  $\delta_B = \delta_T + i(\kappa^{\sharp_g})$ \cite{JR}.  Now, if we chose a bundle-like metric such that the mean curvature  form ia bsic harmonic, then from Proposition 8.11,   $\delta_B\kappa^c=0$.  Since $\kappa^{\sharp_g}$ is the $g_Q$-dual vector to $\kappa$, we get
\begin{align*}
{\rm div}_\nabla(\kappa^{\sharp_g})^c =-\delta_T\kappa^c =-\delta_B\kappa^c + i(\kappa^{\sharp_g})\kappa^c =|\kappa|^2.
\end{align*}
\end{proof}

\begin{cor}  On a transverse  K\"ahler foliation of constant holomorphic sectional curvature $h$, it holds that for any $\varphi\in\Gamma Sp^J_0(\mathcal F)$
\begin{align*}
\mathcal P_{\rm tr}^0\varphi =\nabla_{\rm tr}^*\nabla_{\rm tr}\varphi - \frac{h}4 n(n+1)\varphi+\frac14|\kappa|^2\varphi.
\end{align*}
\end{cor}
\begin{proof}  The proof follows  from (\ref{8-11}) in Theorem 8.10 and Lemma 8.12.
\end{proof}

\begin{thm}  On  a transverse  K\"ahler foliation of constant  holomorphic sectional curvature $h$ with $M$ closed,  any eigenvalue $\lambda$ of  $\mathcal P_{\rm tr}^0$ on $\Gamma Sp^J_0(\mathcal F)$ satisfies
\begin{align*}
\lambda \geq - {h\over 4} n(n+1) +\frac 14 \min |\kappa|^2.
\end{align*}
\end{thm}
\begin{proof}  Let $\mathcal P_{\rm tr}^0\varphi =\lambda\varphi$ for a $\varphi\in\Gamma S^J_0(\mathcal F)$. From Corollary 8.13, we have
\begin{align*}
\int_M\lambda \Vert\varphi\Vert^2 \mu_M= \int_M\Big(\Vert \nabla_{\rm tr}\varphi\Vert^2 - {h \over 4}n(n+1)\Vert\varphi\Vert^2 +\frac14|\kappa|^2\Vert\varphi\Vert^2\Big)\mu_M.
\end{align*}
Hence  
\begin{align*}
\int_M\Big(\lambda + {h\over 4}n(n+1) -\frac14 |\kappa|^2 \Big)\Vert\varphi\Vert^2  \mu_M = \int_M\Vert \nabla_{\rm tr}\varphi\Vert^2  \mu_M.
\end{align*}
From the equation above,
\begin{align*}
0&\leq \int_M\Big(\lambda + {h\over 4}n(n+1) -\frac14 |\kappa|^2 \Big)\Vert\varphi\Vert^2 \mu_M \\
&\leq \int_M\Big(\lambda + {h\over 4}n(n+1) -\frac14 \min |\kappa|^2 \Big)\Vert\varphi\Vert^2 \mu_M,
\end{align*}
which yields the result.
\end{proof}
\begin{thm} Let $(M,\mathcal F,J,g_Q)$ be a transverse  K\"ahler foliation of constant  and nonpositive holomorphic sectional curvature $h$ on a closed manifold $M$. If $\mathcal F$ is minimal, then any eigenvalue $\lambda$ of $\mathcal P_{\rm tr}$ satisfies
\begin{align*}
\lambda \geq -{h\over 4} n(n-1).
\end{align*}
\end{thm}
\begin{proof} 
Let $\mathcal P_{\rm tr}\varphi =\lambda\varphi$.  Since $\mathcal F$ is minimal and $h\leq 0$,  from (\ref{8-8})
\begin{align*}
\int_M\Big(\lambda+ \frac{h}4 n(n-1)\Big) \Vert\varphi\Vert^2 =\int_M \Vert \nabla_{\rm tr}\varphi\Vert^2  - 2h\int_M\Vert \mathcal H^J(\varphi)\Vert^2\geq 0.
\end{align*}
So the proof is completed.
\end{proof}
\begin{rem}   Theorem 8.15  is a generalization of the point foliation version \cite[Propositon 3.4]{Ha2-1} on a  K\"ahler manifold.
\end{rem}


\begin{thebibliography}{[9]}



\bibitem{AJ} 
J. A.  \'Avarez L\'opez and S. D. Jung, \emph{Transversal Hard Lefschetz theorem on transversely symplectic foliations}, arXiv:1910.09263v2, to appear in Quart. J. Math.

\bibitem{BC} L. Bak and A. Czarnecki, \emph{A Remark on the Brylinski conjecture for orbifolds}, J. Aust. Math. Soc. 91 (2011), 1-12.

\bibitem{Bl}
D. E. Blair, {\it Riemannian geometry of contact and symplectic manifolds}, Birkh\"auser, 2010.

\bibitem{BG}
J. Block and E. Getzler, \emph{Quantization of foliations}, Proceedings of the XXth International Conference on Diffenential Geometric Methods in Theoretical Physics (New York, 1991) (Eds. S. Catto and A. Rocha), World Scientific Publishing Co., River Edge, NJ, 1992, 471-487.


\bibitem{Cw}
L. A. Cordero and R. A. Wolak, {\it Examples of foliations with foliated geometric structures}, Pacific J. math. 142 (1990), 265-276.

\bibitem{FO}
G. B. Folland, {\it Harmonic analysis in phase space},  Annals of Mathematics Studies, 122. Princeton, NJ; Princeton University Press (1989).

\bibitem{GR} I. Gelfand, V. Retakh and M. Shubin, {\it Fedosov manifolds}, Adv. in Math. 136 (1998), 104-140


\bibitem{Ha1}
K. Habermann, {\it The Dirac operator on Symplectic spinors}, Anal. Glob. Anal. Geom. 13 (1995), 155-168.

\bibitem{Ha2}
K. Habermann, {\it Basic properties of Symplectic Dirac operators}, Commun. Math. Phys. 184 (1997), 629-652.

\bibitem{Ha2-1}
K. Habermann, {\it Symplectic Dirac operators on K\"ahler manifolds}, Math. Nachr. 211 (2000), 37-62.

\bibitem{Ha3}
K. Habermann and L. Habermann, {\it Introduction to symplectic Dirac operators}, Lecture Notes in Mathematics 1887, Springer-Verlag Berlin, 2006. 

\bibitem{JR}
S. D. Jung and K. Richardson, {\it The mean curvature of transverse K\"ahler foliations}, Documenta math. 24 (2019), 995-1031.


\bibitem{Ka}
M. Kashiwara and M. Vergne, {\it On the Segal-Shale-Weil Representations and Harmonic polynomials}, Invent. Math. 44 (1978), 1-47.


\bibitem{Ko}
B. Kostant, {\it Symplectic spinors}, Symposia Mathematica, Vol. XIV, 1974.

\bibitem{KN}
S. Kobayashi and K. Nomizu, {\it Foundations of differential geometry.} Vol. II. New York-London-Sydney: Interscience Publishers a division of John Wiley and Sons, 1969.

\bibitem{KR}
S. Kr\'ysl, {\it Symplectic killing spinors}, Commentationes Mathematicae Universitatis Carolinae 53 (2012), 19-35.


\bibitem{LM}
H. B. Lawson, Jr. and M. L. Michelsohn, {\it Spin geometry,} Princeton Univ. Press, Princeton, New
Jersey, 1989.


\bibitem{LI}
Y. Lin, \emph{Hodge theory on transversely symplectic foliations}, Quart. J. Math. 69 (2018), 585--609.

\bibitem{Pa} H. K. Pak, {\it Transversal harmonic theory for transversally symplectic flows}, J. Aust. Math. Soc. 84 (2008), 233-245.

\bibitem{To} Ph. Tondeur, \emph{Affine Zusammenh\"ange auf Mannigfaltigkeiten mit fast-symplektischer Struktur}, Comment. Math. Helv. 36(1961), 234-244




\bibitem{Va1} 
I. Vaisman, {\it Symplectic curvature tensors}, Monatsh. Math. 100 (1985), 299-327.

\end{thebibliography}
\end{document}